\RequirePackage{fix-cm}
\documentclass[smallextended]{svjour3}       
\smartqed  
\usepackage{graphicx}
\usepackage{amsmath,amsfonts}
\usepackage{bm}
\usepackage{xcolor}
\usepackage{booktabs}
\usepackage{array}
\usepackage{algorithm}
\usepackage{algorithmic}
\usepackage{multirow}
\usepackage{ctable}
\usepackage{tabu}
\usepackage{setspace}
\usepackage{xspace}
\usepackage[american]{babel}
\newcommand{\CRo}{\mathcal{CR}1}
\newcommand{\CRt}{\mathcal{CR}2}
\newcommand{\V}{\bm{V}}
\newcommand{\F}{\mathcal{F}}
\newcommand{\ip}[2]{\left\langle #1, #2 \right\rangle}
\newcommand{\R}{\mathbb{R}}
\newcommand{\wtilde}[1]{\widetilde{#1}}
\newcommand{\e}{\varepsilon}
\newcommand{\E}{\mathbb{E}}
\newcommand{\Var}{\textrm{Var}}
\DeclareMathOperator{\conv}{conv}

\newcommand{\black}[1]{{\color{black} #1}}
\newcommand{\M}{\bm{M}}
\newcommand{\iin}{\textrm{in}}
\newcommand{\oout}{\textrm{out}}
\newcommand{\nthickrule}{\specialrule{.1em}{.05em}{.05em}}

\DeclareMathOperator*{\argmax}{arg\,max}
\DeclareMathOperator*{\argmin}{arg\,min}
\newcommand{\rspca}{SPCAgs\xspace}
\newtheorem{manualtheoreminner}{Theorem}
\newenvironment{manualtheorem}[1]{%
  \renewcommand\themanualtheoreminner{#1}%
  \manualtheoreminner
}{\endmanualtheoreminner}

\begin{document}

\title{Solving sparse principal component analysis with global support \thanks{A preliminary version of this paper was published in~\cite{wangupper}.}}



\author{Santanu S. Dey         \and
        Marco Molinaro \and 
Guanyi Wang
}


\institute{Santanu S. Dey \at
              ISyE, Georgia Institute of Technology \\
              \email{santanu.dey@isye.gatech.edu}           
           \and
           Marco Molinaro \at
              Computer Science Department, Pontifical Catholic University of Rio de Janeiro\\
\email{mmolinaro@inf.puc-rio.br}
\and
Guanyi Wang \\
ISyE, Georgia Institute of Technology \\
\email{gwang93@gatech.edu}
}

\date{Received: date / Accepted: date}

\maketitle

\begin{abstract}
Sparse principal component analysis with global support (\rspca), is the problem of finding the top-$r$ leading principal components such that all these principal components are linear combinations of a common subset of at most $k$ variables. \rspca is a popular dimension reduction tool in statistics that enhances interpretability compared to regular principal component analysis (PCA). Methods for solving \rspca  in the literature are either greedy heuristics (in the special case of $r = 1$) with guarantees under restrictive statistical models or algorithms with stationary point convergence for some regularized reformulation of \rspca. Crucially, none of the existing computational methods can efficiently guarantee the quality of the solutions obtained by comparing them against dual bounds.

	In this work, we first propose a convex relaxation based on operator norms that provably approximates the feasible region of \rspca within a $c_1 + c_2 \sqrt{\log r} = O(\sqrt{\log r})$ factor for some constants $c_1, c_2$. To prove this result, we use a novel random sparsification procedure that uses the \emph{Pietsch-Grothendieck factorization theorem} and may be of independent interest. We also propose a simpler relaxation that is second-order cone representable and gives a $(2\sqrt{r})$-approximation for the feasible region.
	
	Using these relaxations, we then propose a convex integer program that provides a dual bound for the optimal value of \rspca. Moreover, it also has worst-case guarantees: it is within a multiplicative/additive factor of the original optimal value, and the multiplicative factor is $O(\log r)$ or $O(r)$ depending on the relaxation used.
	
	Finally, we conduct computational experiments that show that our convex integer program provides, within a reasonable time, good upper bounds that are typically significantly better than the natural baselines.
		
\end{abstract}

\section{Introduction}
\label{introduction}

Principal component analysis (PCA) is a popular tool for dimension reduction and data visualization. Given a \textit{sample matrix} $\bm{X} = \left( \bm{x}_1, \ldots, \bm{x}_M \right) \in \mathbb{R}^{d \times M}$ where each column denotes a $d$-dimensional zero-mean sample, the goal is to find the top-$r$ leading eigenvectors $\bm{V} := (\bm{v}_1, \ldots, \bm{v}_r) \in \mathbb{R}^{d \times r}$ (\textit{principal components}), namely the matrix satisfying 
\begin{align*}
	\argmax_{\bm{V}^{\top} \bm{V} = \bm{I}^r}\text{Tr} \left( \bm{V}^{\top} \bm{A} \bm{V} \right), \tag{PCA} \label{eq:PCA}
\end{align*}
where $\text{Tr}(\cdot)$ is the trace, $\bm{A} := \frac{1}{M} \bm{X} \bm{X}^{\top}$ is the \textit{sample covariance matrix}, and $\bm{I}^r$ denotes the $r \times r$ identity matrix. 

Principal components usually tend to be dense{\color{blue};} that is, the principal components usually involve almost all the $d$ variables/attributes. This leads to a lack of interpretability of the results from PCA, especially in the high-dimensional setting, e.g., clinical analysis, biological gene analysis, and computer vision~\cite{burgel2010clinical,yeung2001principal,jolliffe2016principal}.  Moreover, anecdotally, principal component analysis is also known to generate large generalization errors and therefore makes inaccurate predictions.  To enhance the interpretability and reduce the generalization error, it is natural to consider alternatives to PCA where a sparsity constraint is incorporated. There are different choices of sparsity constraints depending on the context and application. 

In this paper, we consider the \textit{Sparse PCA with global support} (\rspca) problem (see, for example~\cite{vu2012minimax}) defined as follows: Given a sample covariance matrix $\bm{A} \in \mathbb{R}^{d \times d}$ and a \textit{sparsity parameter} $k ~ (\leq d)$, the task is to find out the top-$r$ $k$-sparse principal components $\bm{V} \in \mathbb{R}^{d \times r}$ ($r \leq k$) given by 
\begin{align*}
	\argmax_{\bm{V}^{\top} \bm{V} = \bm{I}^r, ~ \|\bm{V}\|_0 \leq k}\text{Tr} \left( \bm{V}^{\top} \bm{A} \bm{V} \right),  \tag{\rspca} \label{eq:Multi-SPCA}
\end{align*}
where the \textit{row-sparsity constraint} $\|\bm{V}\|_0 \leq k$ denotes that there are at most $k$ non-zero rows in the matrix $\bm{V}$, i.e., the principal components share \textit{global support}. 

\subsection{Literature review}
There is extensive literature on (approximately) solving variations of the sparse PCA problem, and existing approaches can be broadly classified into the following five categories.

In the first category, instead of dealing with the non-convex sparsity constraint directly, the papers \cite{jolliffe2003modified,zou2006sparse,attouch2010proximal,ma2013alternating,vu2013fantope,bolte2014proximal,erichson2018sparse,chen2019alternating} incorporate additional regularizers to the objective function to enhance the sparsity of the solution. Similar to LASSO for the sparse linear regression problem, these new formulations can be optimized via alternating-minimization type algorithms. We note here that the optimization problem presented in \cite{jolliffe2003modified} is NP-hard to solve, and there is no convergence guarantee for the alternating-minimization method given in \cite{zou2006sparse}. The papers \cite{attouch2010proximal,ma2013alternating,vu2013fantope,bolte2014proximal,erichson2018sparse,chen2019alternating} propose their own formulations for the sparse PCA problem, and show that the alternating-minimization algorithm converges to stationary (critical) points. However, the solutions obtained using the above methods cannot guarantee the row-sparsity constraint $\|\bm{V}\|_0 \leq k$. Moreover, none of these methods are able to provide worst-case guarantees for the quality of the solutions obtained.

The second category of methods works with convex relaxations of the sparsity constraint. A majority of this work is for solving  \ref{eq:Multi-SPCA} for the case where $r =1$. The papers \cite{d2005direct,d2008optimal,zhang2012sparse,d2014approximation,kim2019convexification,yongchun2020exact} directly incorporate the sparsity constraint (for $r = 1$ case) and then relax the resulting optimization problem into some convex optimization problem -- usually a semi-definite programming (SDP) relaxation. However, SDPs are often difficult to scale to large instances in practice. Still for the $r=1$ case, the paper~\cite{dey2018convex} proposes a framework to find dual (upper) bounds for \rspca using convex quadratic integer programs. 

A third category of papers presents fixed parameter tractable exact algorithms, where the fixed parameters are usually the rank of the data matrix $\bm{A}$ and $r$. The paper \cite{papailiopoulos2013sparse} proposes an exact algorithm to find the global optimal solution of \ref{eq:Multi-SPCA} with $r = 1$ with running-time of $O(d^{\text{rank}(\bm{A}) + 1} \log d)$. Later the paper \cite{asteris2015sparse} gave a combinatorial method for another variant of sparse PCA with \textit{disjoint} supports, i.e., if $\bm{V}_{j_1}, \bm{V}_{j_2}$ are two columns of $\bm{V}$, then $\text{supp}(\bm{V}_{j_1}) \cap \text{supp}(\bm{V}_{j_2}) = \emptyset$ for $j_1 \neq j_2 \in [r]$. They show that their algorithm outputs a $(1 - \epsilon)$-approximation in time  polynomial in the data dimension $d$ and the reciprocal of $\epsilon$, but exponential both in the rank of the sample covariance matrix $\bm{A}$ and in $r$. Recently \cite{alberto2019sparse} provides a general method for solving \ref{eq:Multi-SPCA} exactly with computational complexity polynomial in $d$, but exponential in $r$ and $\text{rank}(\bm{A})$. The paper \cite{alberto2019sparse} mentions that the results obtained are of theoretical nature, and these methods may not be practically implementable.

A fourth category of results is that of specialized iterative heuristic methods for finding good feasible solutions of  \ref{eq:Multi-SPCA}~\cite{sigg2008expectation,johnstone2009sparse,mackey2009deflation,journee2010generalized,probel2011technical,boutsidis2011sparse,asteris2011sparse,yuan2013truncated,papailiopoulos2013sparse}. In particular, many papers focus on thresholding methods for sparse PCA, with one of the earliest papers being~\cite{johnstone2009sparse}. The paper~\cite{probel2011technical} presents guarantees for the joint thresholding method and argues that this method explains as much variance as from more sophisticated algorithms. To the best of our understanding, there is no natural way to generalize most of these methods for solving \ref{eq:Multi-SPCA} when $r >1$ (except the results in~\cite{probel2011technical}, which hold for $r >1$).

The final category of papers presents algorithms that perform well under the assumption of a statistical model.  Under the assumption of an underlying statistical model, the paper \cite{gu2014sparse} presents a family of estimators for \ref{eq:Multi-SPCA} $r=1$ with the so-called `oracle property' via solving a semidefinite relaxation of sparse PCA. The paper \cite{deshpande2016sparse} analyzes a covariance thresholding algorithm (first proposed by \cite{krauthgamer2015semidefinite}) for the $r = 1$ case. They show that under a specific statistical model, this algorithm correctly recovers the support of the underlying true solution with high probability when the sparsity parameter $k$ is at most of order $\sqrt{M}$, with $M$ being the number of samples. This sample complexity, combined with the lower bounds from \cite{berthet2013computational,ma2015sum}, suggest that no polynomial-time algorithm can do significantly better under their statistical assumptions.  There is also a series of papers \cite{vu2012minimax,cai2013sparse,wang2014tighten,cai2015optimal,lei2015sparsistency} that provide the minimax rate of estimation for sparse PCA. However, all these papers require underlying statistical models and thus do not have worst-case guarantees in the model-free case, which is the focus of this paper. 


\subsection{Our contributions}

In this paper, we propose explicit convex relaxations for the sparse PCA problem (\rspca) with provable quality of their approximations. We use them to obtain a convex integer program that provides a dual bound for the optimal value of \rspca, also with provable guarantees. Finally, we computationally evaluate the the viability and quality of the bounds obtained for fairly large instances. We present details below.

\paragraph{Explicit convex relaxations of the feasible region.} Let $$\mathcal{F} := \{\bm{V}\,|\, \bm{V}^{\top} \bm{V} = \bm{I}^r, ~ \|\bm{V}\|_0 \leq k\}$$ denote the feasible region of \eqref{eq:Multi-SPCA} and let $\text{opt}^{\mathcal{F}}(\bm{A})$ denote the optimal value of \eqref{eq:Multi-SPCA} for the sample covariance matrix $\bm{A}$.  

 Note that the objective function of \ref{eq:Multi-SPCA} is that of maximizing a convex function, and so at least one of the extreme points of the feasible region $\mathcal{F}$ is an optimal solution. Hence, it is important to approximate the convex hull of the feasible region well. Our first explicit convex relaxation for $\mathcal{F}$ has constraints based on the operator norms of the matrix $\V$ and is given by 
	\begin{align}
	\CRo := \left\{\bm{V} \in \mathbb{R}^{d \times r} \, \left|\, \|\bm{V}\|_{\textup{op}} \le 1,~\|\bm{V}\|_{2 \rightarrow 1} \le \sqrt{k},~ \sum_{i =1}^d \|\bm{V}_{i,:}\|_2 \le \sqrt{rk} \right.\right\},  \label{eq:CR1} \tag{$\CRo$}
	\end{align}
	see Section~\ref{sec:notation} for the required definitions. Importantly, we prove that this relaxation is within a factor of $c_1 + c_2\sqrt{\log r} =  O(\sqrt{\log r})$, for some constants $c_1, c_2 > 0$, of the convex hull of $\mathcal{F}$.
	
	\begin{theorem}\label{thm:cr1}
	For every positive integers $d,r,k$ such that $1 \leq r \leq k \leq d$ the convex relaxation $\CRo$ satisfies 
\begin{align*}
\mathcal{F} ~\subseteq~ \mathcal{CR}1 ~\subseteq~ \rho_{\CRo} \cdot \textup{conv}\left(\mathcal{F}\right)
\end{align*}
for $\rho_{\CRo} = 2 + \max\{6 \sqrt{2 \pi},\, 12 \sqrt{\log 50r}\} = O(\sqrt{\log r})$. 
	\end{theorem}
	The proof of this result is presented in Section~\ref{sec:convex-relax1}.

Given the simplicity of this relaxation and its provable approximation guarantee, $\CRo$ seems to be an important set not only for sparse PCA but also for other problems with row-sparsity constraints. To prove this result, we use a novel randomized matrix sparsification procedure that, given a matrix $\V^*$ in $\CRo$ produces a row-sparse matrix $\V$ with a controlled spectral norm (hence in $\F$) that is close to the starting point $\V^*$. The main difficulty is effectively leveraging the information provided by the simple formulation $\CRo$, mainly the control of the $\|\cdot\|_{2 \rightarrow 1}$ norm. For that, we employ in our sparsification a row sampling procedure where the weights of the rows are given by the \emph{Pietsch-Grothendieck factorization theorem}~\cite{pietsch1978operator}. We believe this idea may also find uses in other problems with row-sparsity constraints.

It is known that it is NP-hard to compute the $\ell_{2 \rightarrow 1}$-norm present in the constraints of $\mathcal{CR}1$~\cite{steinberg2005computation}. However, it is possible to approximate the $\ell_{2 \rightarrow 1}$-norm constraint within a constant-factor using a semi-definite relaxation proposed in \cite{tropp2009column}. The resulting SDP-representable convex relaxation is as follows:
\begin{align*}
    \mathcal{CR}1' := \left\{ \bm{V} \in \mathbb{R}^{d \times r} \left| 
    \begin{array}{rcl}
	\textup{there exists } \bm{h} \in \mathbb{R}^d_{+}, &\textup{s.t.}&\\
        \sum_{i = 1}^d \|\bm{V}_{i, :}\|_2 &\leq& \sqrt{rk} \\
        \sum_{i =1}^d \bm{h}_i &\leq& \frac{\pi}{2} k \\
       \left[\begin{array}{rl} \bm{I}^d& \bm{V} \\ \bm{V}^{\top}& \bm{I}^r \end{array}\right] &\succeq& \bm{0}^{d +r, d + r} \\
       \left[\begin{array}{rl} \textup{diag}(\bm{h})& \bm{V} \\ \bm{V}^{\top}& \bm{I}^r \end{array}\right] &\succeq& \bm{0}^{d +r, d + r} \\
    \end{array}
    \right.\right\},
\end{align*}
where $\textup{diag}(\bm{h})$ is the diagonal matrix whose diagonal elements are given by the vector $h$, $\bm{I}^d$ is the $d\times d$ identity matrix,  and $\bm{0}^{d +r, d + r}$ is the all zeros matrix of size $(d + r) \times (d +r)$.
The convex relaxation $\mathcal{CR}1'$ is also guaranteed to be within a multiplicative ratio of $O(\sqrt{\log r})$.
\begin{theorem} \label{thm:cr1'}
    For every positive integers $d, r, k$ such that $1 \leq r \leq k \leq d$, the SDP-representable relaxation $\mathcal{CR}1'$ satisfies 
    \begin{align*}
        \text{conv}(\mathcal{F}) \subseteq \mathcal{CR}1' \subseteq \rho_{\mathcal{CR}1'} \cdot \text{conv}(\mathcal{F}),
    \end{align*}
		 for {$\rho_{\mathcal{CR}1'} = \sqrt{\frac{\pi}{2}} \cdot \rho_{\CRo} = O(\sqrt{\log r})$}.    
\end{theorem}
See Section~\ref{sec:SDR-L21-norm} for a proof of Theorem~\ref{thm:cr1'}. 

While this relaxation is easier to handle computationally compared to $\mathcal{CR}1$, the semi-definite constraints are still typically hard to solve in practice for large instances. Therefore, we next present a simpler convex relaxation of $\mathcal{F}$ that, while having a worse worst-case guarantee, is second-order cone representable and hence can be more efficiently optimized over using interior-point methods.

	
	\begin{theorem} \label{thm:cr2}
		For every $d,r,k$ positive integers such that $1 \leq r \leq k \leq d$, there is a relaxation $\mathcal{CR}2$ (see \eqref{eq:CR2}) that is second-order cone representable and has the guarantee
		\begin{align*}
			\conv(\mathcal{F}) \subseteq \mathcal{CR}2 \subseteq \rho_{\mathcal{CR}2} \cdot \conv(\mathcal{F}),
		\end{align*}
	where $\rho_{\mathcal{CR}2} \leq 2\sqrt{r}$. 
	\end{theorem}		
	This generalizes the main theoretical result in~\cite{dey2018convex} for the case $r =1$. See Section~\ref{sec:convex-relax2} for a proof of Theorem~\ref{thm:cr2}.

\paragraph{Convex integer programming formulation for obtaining dual bounds.} While the above relaxations allow us to convexity the feasible region, since the objective function of \ref{eq:Multi-SPCA} is also non-convex (i.e., maximizing a convex function), they do not yet give ``full relaxations'' that yield tractable upper (dual) bounds for the problem. Recall that dual bounds are typically crucial for effective computational procedures for non-convex problems, which are usually solved using branch-and-bound type algorithms that use dual bounds to prune the solution space.

	To handle the non-convex objective function, we consider the natural approach of upper bounding the objective function by piecewise linear functions, which can be modeled using binary variables and special ordered sets (SOS-II)~\cite{wolsey1999integer} (see Section~\ref{sec:plwconstruction} for the construction). Used together with a convex relaxation $\CRo$, $\CRo'$, or $\CRt$, this gives a convex integer programming relaxation for \ref{eq:Multi-SPCA}.
	
	Interestingly, we show that this full relaxation also has a provable approximation guarantee, providing a dual bound within a multiplicative/additive factor of the optimal value of the original problem. 
	
\begin{theorem} \label{thm:CIP}
Let $\textup{opt}^{\mathcal{F}}$ be the optimal value of \ref{eq:Multi-SPCA}. Then there is a convex integer program \eqref{eq:CIP} using the relaxation $\CRo$, $\mathcal{CR}1'$, or $\CRt$ whose optimal value $\textup{ub}^{\mathcal{CR}i}$ satisfies the following:
\begin{align*}
	\textup{opt}^{\mathcal{F}}(\bm{A}) \leq \textup{ub}^{\mathcal{CR}i} \leq \rho_{\mathcal{CR}i}^2 \cdot \textup{opt}^{\mathcal{F}}(\bm{A}) + \textup{additive-term}(\bm{A}),
\end{align*}
	where the $\textup{additive term}$ depends on the input matrix $\bm{A}$ and the parameters used in the piecewise linear approximation of the objective function, and $\rho_{\mathcal{CR}i}$ is the approximation ratio guarantee from Theorems \ref{thm:cr1}, \ref{thm:cr1'} and \ref{thm:cr2}. 
\end{theorem}	
See Section~\ref{sec:ub-PLA} for a proof of Theorem~\ref{thm:CIP}.
	

	We remark that despite having an additive term, this bound is invariant to rescaling of the data matrix $\bf{X}$, see Appendix \ref{app:invariance} for details. Moreover, the additive term is bounded from above by 
$$\textup{add}(\bm{A}) \leq \textup{Tr}(\bm{A})\cdot \frac{r}{4N^2},$$ 
where $N$ is the number of pieces used in the piecewise approximation of each quadratic term in the objective function (See Section~\ref{sec:ub-PLA} for details). For a required level of additive-term error (and thus a required \emph{effective multiplicative ratio}), one can select an appropriate value of $N$. As $N \rightarrow \infty$ this becomes a purely multiplicative guarantee with approximation factor $\rho_{\mathcal{CR}i}^2$.

\paragraph{Computational experiments.} In order to evaluate both its practical viability and the quality of the dual bound it provides, we perform computational experiments with the full relaxation provided by our proposed convex integer program (Theorem \ref{thm:CIP}). We use both synthetic and real data with up to 2000 features, which yield input matrices of size $2000 \times 2000$. See Section~\ref{sec:numerical-exp} for details of all our numerical experiments.

Solving our convex integer program for the larger instances requires care, and we employ cutting planes and a submatrix splitting technique to speed up the computations. See Section~\ref{sec:num-sub-matrix} and Appendix~\ref{app:reduce-running-time} for these details.  

Obtaining the good feasible solutions required to evaluate the quality of our dual bounds poses a challenge given the lack of heuristics for \ref{eq:Multi-SPCA} when $r > 1$ and the size of the instances. To mitigate this, we use an optimized version of the natural greedy heuristic that looks for which rows of $\V$ should be non-zero and sets values for them based on the standard PCA problem restricted to these variables. The heuristic greatly reduces the number of eigenvalue computations required, which is a bottleneck of the process. See Section~\ref{sec:primal-bound} for a presentation of this heuristic.

 The numerical results show that our convex integer program provides within a reasonable time good upper bounds that are typically significantly better than an SDP relaxation and another baseline.
	
\paragraph{Note.} A preliminary version of this paper was published in~\cite{wangupper}. The current version has many new results: in particular, $\mathcal{CR}1$, $\mathcal{CR}1'$, and the results on their strengths are completely new, and the numerical experiments have been completely revamped.


\subsection{Notation}\label{sec:notation}
We use regular lower case letters, for example $\alpha$, to denote scalars. For a positive integer $n$, let $[n] := \{1, \ldots, n\}$. For a set $S \subseteq \mathbb{R}^n$ and a $\rho > 0$ denote $\rho \cdot S := \{ \rho x \,|\, x \in S \}$.

We use bold lower case letters, for example $\bm{a}$, to denote vectors. We denote the $i$-th component of a vector $\bm{a}$ as $\bm{a}_{i}$.  Given two vectors, $\bm{u}, \bm{v} \in \mathbb{R}^n$, we represent the inner product of $\bm{u}$ and $\bm{v}$ by $\langle \bm{u}, \bm{v}\rangle$. Sometimes it will be convenient to represent the outer product of vectors  using $\otimes$, i.e., given two vectors $\bm{a}, \bm{b} \in \mathbb{R}^n$, $\bm{a} \otimes \bm{b}$ is the matrix where $[\bm{a} \otimes \bm{b}]_{i,j} = \bm{a}_i \bm{b}_j$. We denote the unit vector in the direction of the $j$th coordinate as $\bm{e}^j$.

We use bold upper case letters, for example $\bm{A}$, to denote matrices. We denote the $(i,j)$-th component of a matrix $\bm{A}$ as $\bm{A}_{ij}$. We use $\text{supp}(\bm{A}) \subseteq [m]$ to denote the indices of the non-zero rows of a matrix $\bm{A} \in \R^{m \times n}$. We use regular upper case letters, for example $I$, to denote the set of indices. Given any matrix $\bm{A} \in \mathbb{R}^{m \times n}$ and $I \subseteq [m], J \subseteq [n]$, we denote the sub-matrix of $\bm{A}$ with rows in $I$ and columns in $J$ as $\bm{A}_{I,J}$. For $I \subseteq [m]$, to simplify notation we use the Matlab-like notation to denote the submatrix of $\bm{A} \in \mathbb{R}^{m \times n}$ corresponding to the rows with indices in $I$ as $\bm{A}_{I, :}$ (instead of $\bm{A}_{I, [n]}$). Similarly for $i \in [m]$, we denote the $i^{\textup{th}}$ row of $\bm{A}$ as $\bm{A}_{i,:}$. Also, for $J \subseteq [n]$ we denote the submatrix of $\bm{A} \in \mathbb{R}^{m \times n}$ corresponding to the columns with indices in $J$ as $\bm{A}_{:, J}$ (instead of $\bm{A}_{[m], J}$),  and for $j \in [n]$, we denote the $j^{\textup{th}}$ column of $\bm{A}$ as $\bm{A}_{:, j}$.

For a symmetric square matrix $\bm{A}$, we denote the largest eigenvalue of $\bm{A}$ as $\lambda_{\textup{max}}(\bm{A})$. Given two symmetric matrices $\bm{A}, \bm{B} \in \mathbb{R}^{n \times n}$, we say that $\bm{A} \preceq \bm{B}$ if $\bm{B} - \bm{A}$ is a positive semi-definite matrix. Given $\bm{U}, \bm{V} \in \mathbb{R}^{m\times n}$, we let $\langle \bm{U}, \bm{V} \rangle = \sum_{i = 1}^m\sum_{j = 1}^n \bm{U}_{ij} \bm{V}_{ij}$ to be the inner product of matrices. We use $\bm{0}^{p, q}$ to denote the matrix of size $p \times q$ with all entries equal to zero, and use $\bm{I}^{p}$ to denote the identity matrix of size $p \times p$. We use $\oplus$ to denote the direct sum of matrices, i.e., given matrices $\bm{A} \in \mathbb{R}^{p \times q}, \bm{B} \in \mathbb{R}^{m \times n}$, $$\bm{A} \oplus \bm{B} := \left[\begin{array}{cc}\bm{A} & \bm{0}^{p, n}\\ \bm{0}^{m,q} & \bm{B} \end{array}\right].$$ The operator norm $\|\bm{A}\|_{p \rightarrow q}$ of a matrix $\bm{A} \in \mathbb{R}^{m \times n}$ is defined as
\begin{align*}
\|\bm{A}\|_{p \rightarrow q}:= \textup{max}_{\bm{x} \in \mathbb{R}^n, \|\bm{x}\|_p = 1}\|\bm{Ax}\|_q.
\end{align*}
We sometimes refer to $\|\bm{A}\|_{2 \rightarrow 2}$ as $\|\bm{A}\|_{\textup{op}}$. Note that $\|\bm{A}\|_{\textup{op}}$ is the largest singular value of $\bm{A}$. The Frobenius norm of a matrix $\bm{A}$ is denoted as $\|\bm{A}\|_F$. We use the notation $\| \bm{A}\|_{1}$ to be the sum of absolute values of the entries of $A$.


{\color{black}
\section{Convex relaxations of $\mathcal{F}$} \label{sec:convex-relax}



\subsection{Operator-norms relaxation $\mathcal{CR}1$} \label{sec:convex-relax1}

	In the vector case, i.e., $r=1$, a natural convex relaxation for $\F$ is to control the sparsity via the $\ell_2$ and $\ell_1$ norms, namely to consider the set $\{\bm{v} \in \R^d \mid \| \bm{v}\|_2 \le 1,~\|\bm{v}\|_1 \le \sqrt{k}\}$ (see~\cite{dey2018convex}). It is easy to see that this is indeed a relaxation in the case $r=1$: if $\bm{v} \in \F$, then by definition $\ip{\bm{v}}{\bm{v}} = 1$ and so $\|\bm{v}\|_2 =1$, and since $\bm{v}$ is a $k$-sparse vector we get, using the standard $\ell_1$ vs $\ell_2$-norm comparison in $k$-dimensional space, $\|\bm{v}\|_1 \le \sqrt{k} \cdot \|\bm{v}\|_2 = \sqrt{k}$. 

	Here we consider the relaxation $\CRo$ that generalizes this idea for any $r$, which we now recall:
	\begin{align*}
	\CRo := \left\{\bm{V} \in \mathbb{R}^{d \times r} \, \left|\, \|\bm{V}\|_{\textup{op}} \le 1,~\|\bm{V}\|_{2 \rightarrow 1} \le \sqrt{k},~ \sum_{i =1}^d \|\bm{V}_{i,:}\|_2 \le \sqrt{rk} \right.\right\}.
\end{align*}
	Thus we now use both the $\ell_{2\rightarrow 1}$ norm and the sum of the lengths of the rows of $\bm{V}$ to take the role of the $\ell_1$-norm proxy for sparsity (by convexity of norms both constraints are convex). While is it not hard to see that this is a relaxation of $\F$, we further show that it has a provable approximation guarantee.


	For the remainder of the section we prove this approximation guarantee of $\CRo$, namely that 
	\begin{align*}
	\mathcal{F} ~\subseteq~ \mathcal{CR}1 ~\subseteq~ \rho_{\CRo} \cdot \textup{conv}\left(\mathcal{F}\right)
  \end{align*}
for $\rho_{\CRo} = 2 + \max\{6 \sqrt{2 \pi},\, 12 \sqrt{\log 50r}\}$, thus proving Theorem \ref{thm:cr1}. We prove each of these inclusions separately. 	 


\subsubsection{Proof of first inclusion in Theorem \ref{thm:cr1}: $\mathcal{F} \subseteq \mathcal{CR}1$}\label{sec:cr1valid}

Consider any matrix $\bm{V}$ in $\F = \{\bm{V} \in \mathbb{R}^{d \times r} \,|\, \bm{V}^{\top} \bm{V} = \bm{I}^r, ~ \|\bm{V}\|_0 \leq k \}$, we show that such $\bm{V}$ satisfies the constraints of $\CRo$. 

\begin{itemize} \itemsep8pt
    \item \textbf{1st constraint of $\CRo$.} Observe that 
    \begin{align*}
        \|\bm{V}\|_{\textup{op}} = \textup{max}_{\bm{x} \in \mathbb{R}^{{r}}, \|\bm{x}\|_2 = 1}\|\bm{Vx}\|_2 = & ~ \textup{max}_{\bm{x} \in \mathbb{R}^{{r}}, \|\bm{x}\|_2 = 1} \sqrt{\langle \bm{V} \bm{x}, \bm{V}\bm{x} \rangle} \\
        = & ~ \textup{max}_{\bm{x} \in \mathbb{R}^{{r}}, \|\bm{x}\|_2 = 1} \sqrt{\langle \bm{x}, \bm{V} ^{\top}\bm{V}\bm{x} \rangle} = 1
    \end{align*}
    Therefore, we obtain that for any $\bm{V} \in \mathcal{F}$, we have $\|\bm{V} \|_{\textup{op}} \leq 1$. 
    
    \item \textbf{2nd constraint of $\CRo$.} By the definition of $\|\cdot\|_{2 \rightarrow 1}$, it is equivalent to verify that $\|\bm{V} \bm{x}\|_1 \leq \sqrt{k}$ for all $\bm{x} \in \mathbb{R}^r$ such that $\| \bm{x}\|_2 \leq 1$. Since $\bm{V}$ is $k$-row-sparse, $\bm{Vx}$ is a $k$-sparse vector and hence by $\ell_1$ vs $\ell_2$-norm comparison in $k$-dim space we get
    \begin{align*}
        \|\bm{Vx}\|_1 \le \sqrt{k} \cdot \|\bm{Vx}\|_2 \leq \sqrt{k}, 
    \end{align*}
    where the last inequality follows from the fact that $\|\bm{Vx}\|_2 \leq \|\bm{V}\|_{\textup{op}}$ for all $\bm{x}$ satisfying $\|\bm{x}\|_2 \leq 1$. 
    
    \item \textbf{3rd constraint of $\CRo$.} Since $\|\bm{V}\|_{\text{op}} \leq 1$, then each column of $\bm{V}$ (i.e., $\bm{V}_{:, j}$ for all $j = 1, \ldots, r$) has a $\ell_2$-norm at most 1. Since $\bm{V}$ has $r$ columns, then 	
    \begin{align*}
	    r \geq \sum_{j = 1}^r \|\bm{V}_{:,j}\|_2^2 = \|\bm{V}\|_F^2 = \sum_{i = 1}^d \|\bm{V}_{i,:}\|_2^2.
    \end{align*}
    The $k$-row-sparse property of $\bm{V} \in \F$ implies that at most $k$ terms $\|\bm{V}_{i,:}\|_2^2$ in the right-hand side of the above inequality are non-zero. Then again applying the $\ell_1$ vs $\ell_2$-norm comparison in $k$-dim space we get
	\begin{align*}
		\sum_{i =1}^d \|\bm{V}_{i,:}\|_2 ~\le~ \sqrt{k} \cdot \sqrt{\sum_{i = 1}^d \|\bm{V}_{i, :}\|_2^2} ~\le~ \sqrt{k} \cdot \sqrt{r} ~ = ~ \sqrt{rk} \, ,
	\end{align*}
	and so the third constraint of $\CRo$ is satisfied. 
\end{itemize}


\subsubsection{Proof of second inclusion in Theorem \ref{thm:cr1}: $\mathcal{CR}1 \,\subseteq\, \rho_{\CRo}   \cdot \conv(\mathcal{F})$} \label{sec:cr1blowup}

	We assume that $k \ge 40$, otherwise $r \le k < 40$ and the result follows from Theorem \ref{thm:cr2}. We prove the desired inclusion by comparing the support function of these sets (Proposition C.3.3.1 of~\cite{hiriart2012fundamentals}), namely we show that for every matrix $\bm{C} \in \R^{d \times r}$
	\begin{align}
		\max_{\bm{V} \in \CRo} \ip{\bm{C}}{\bm{V}} ~\le~ \rho_{\CRo} \cdot \max_{\bm{V} \in \conv(\F)} \ip{\bm{C}}{\bm{V}}. \label{eq:PG0}
	\end{align}

	It will suffice to prove the following sparsification result for the optimum of the left-hand side. 
	
	\begin{lemma} \label{lemma:PGmain}
	Assume $k \ge 40$. Consider $\bm{C} \in \R^{d \times r}$ and let $\bm{V}^*$ be a matrix attaining the maximum on the left-hand side of \eqref{eq:PG0}, namely $\bm{V}^* \in \argmax_{\bm{V} \in \CRo} \ip{\bm{C}}{\bm{V}}$. Then there is a matrix $\bm{V}$ with the following properties:
	\begin{enumerate}
		\item (Operator norm) $\|\bm{V}\|_{op} \le 1 + \max\{\sqrt{18 \pi},\, 6 \sqrt{\log 50r}\}$
		
		\item (Sparsity) $\bm{V}$ is $k$-row-sparse, namely $\|\bm{V}\|_0 \le k$
		
		\item (Value) $\ip{\bm{C}}{\bm{V}} \ge \frac{1}{2}\,\ip{\bm{C}}{\bm{V}^*}$.
	\end{enumerate}
	\end{lemma}
	
	Indeed, if we have such a matrix $\bm{V}$ then $\frac{\bm{V}}{\|\bm{V}\|_{op}}$ belongs to the sparse set $\F$ and has value $\ip{\bm{C}}{\frac{\bm{V}}{\|\bm{V}\|_{op}}} \ge \frac{1}{2 \cdot (1 + \max\{\sqrt{18 \pi},\, 6 \sqrt{\log 50r}\})}\cdot \ip{\bm{C}}{\bm{V}^*}$, showing that \eqref{eq:PG0} holds. 
	
	For the remainder of the section, we prove Lemma \ref{lemma:PGmain}. The idea is to randomly sparsify $\bm{V}^*$ while controlling the operator norm and value. A standard procedure is to sample the rows of $\bm{V}^*$ with probability proportional to their squared length (see~\cite{kannan2017randomized} for this and other sampling methods). However, these more standard methods do not seem to effectively leverage the information that $\|\bm{V}^*\|_{2 \rightarrow 1} \le \sqrt{k}$. Instead, we use a novel sampling more adapted to the $\ell_{2\rightarrow 1}$-norm based on a weighting of the rows of $\bm{V}^*$ given by the so-called \emph{Pietsch-Grothendieck factorization}~\cite{pietsch1978operator}. We state it in a convenient form that follows by applying Theorem 2.2 of~\cite{tropp2009column} to the transpose.

\begin{theorem}[Pietsch-Grothendieck factorization] \label{lemma:PG-factorization} Any matrix $\bm{V} \in \mathbb{R}^{d \times r}$ can be factorized as $\bm{V} = \bm{W} \bm{T} $ of size $\bm{T} \in \mathbb{R}^{d \times r}, ~ \bm{W} \in \mathbb{R}^{d \times d}$, where
	\begin{itemize}
		\item $\bm{W}$ is a nonnegative, diagonal matrix with $\sum_i \bm{W}_{ii}^2 = 1$
		\item $\|\bm{V}\|_{2 \rightarrow 1} \leq \|\bm{T}\|_{\text{op}} \leq \sqrt{\pi / 2} \cdot \|\bm{V}\|_{2 \rightarrow 1}.$
	\end{itemize}
\end{theorem}

	So first apply this theorem to obtain a decomposition $\bm{V}^* = \bm{W} \bm{T}$. Notice that this means the $i$th row of $\bm{V}^*$ is just the $i$th row of $\bm{T}$ multiplied by the weight $\bm{W}_{ii}$. Define the ``probability''
	\begin{align*}
		p_i := \min\left\{ \frac{k}{6} \bigg(\bm{W}_{ii}^2 + \frac{\|\bm{V}^*_{i,:}\|_2}{\sum_{i' = 1}^d \|\bm{V}^*_{i',:}\|_2} \bigg), ~ 1 \right\},
	\end{align*}
	where the minimization between the first term and 1 is to make it a bonafide probability.
	\footnote{For some intuition: The first term in the parenthesis controls the variance of $\wtilde{\bm{W}}_{ii}$, which is $\Var(\wtilde{\bm{W}}_{ii}) \le \frac{\bm{W}^2_{ii}}{p_i} \le \frac{6}{k}$; the second term controls the largest size of a row of $\wtilde{\bm{V}}$, which is $\|\wtilde{\bm{V}}_{i,:}\|_2 \le \left\| \frac{\bm{V}^*_{i,:}}{p_i} \right \|_2 \le \frac{6}{k} \sum_{i'} \|\bm{V}^*_{i',:}\|_2$, which is at most $6\sqrt{\frac{r}{k}} \le 6$ because $\bm{V}^* \in \CRo$.} We then randomly sparsify $\bm{V}^*$ by keeping each row $i$ with probability $\bar{p}_i$ and normalizing it, i.e.,
	 
	\begin{itemize} \itemsep6pt
		\item Compute the random diagonal matrix $\wtilde{\bm{W}}$ with $\wtilde{\bm{W}}_{ii} := \e_i\,\frac{\bm{W}_{ii}}{p_i},$ and $\e_i$ (the indicator that we keep row $i$) takes value 1 with probability $p_i$ and 0 with probability $1-p_i$ (and the $\e_i$'s are independent).
	
		\item Set the random matrix $\wtilde{\bm{V}} := \wtilde{\bm{W}} \bm{T}$, where $\wtilde{\bm{W}}$ is obtained as above and $\bm{T}$ is still the one from the Pietsch-Grothendieck factorization of $\bm{V}^*$.
	\end{itemize}
	Based on the above construction, $\E \wtilde{\bm{W}} = \bm{W}$, and hence this procedure is unbiased: $\E \wtilde{\bm{V}} = \bm{V}^*$. We now show that $\wtilde{\bm{V}}$ satisfies each of the desired items from Lemma \ref{lemma:PGmain} with good probability, and then use a union bound to exhibit a matrix that proves the lemma. 


	\paragraph{Sparsity.} The number of rows $\|\wtilde{\bm{V}}\|_0$ of $\wtilde{\bm{V}}$ is precisely $\sum_{i =1}^d \e_i$, whose expectation is
	\begin{align*}
		\sum_{i =1}^d p_i ~{\le}~ \frac{k}{6} \bigg(\sum_i \bm{W}_{ii}^2 ~+~ 1 \bigg) ~=~ \frac{k}{3}.
	\end{align*}
	Employing the multiplicative Chernoff bound (Lemma \ref{lemma:chernoff} in Appendix \ref{app:conc}) we get 
	\begin{align}
		\Pr\bigg(\|\wtilde{\bm{V}}\|_{0} > k \bigg) \,\le\, \bigg(\frac{2e}{6}\bigg)^k \,<\, \frac{1}{50},  \label{eq:PGsparsity}
	\end{align}
	where the last inequality uses that $k \ge 40$. 
	

	\paragraph{Operator norm.} 
	Let $I$ be the indices $i$ where $p_i < 1$, and let the complement of $I$ be $I^c = [d] \setminus I$ (so $p_i = 1$ and hence $\wtilde{\bm{V}}_{i,:} = \bm{V}^*_{i,:}$ for all $i \in I^c$). From the triangle inequality we can see that $\|\wtilde{\bm{V}}\|_{\textup{op}} \le \|\wtilde{\bm{V}}_{I,:}\|_{\textup{op}} + \|\wtilde{\bm{V}}_{I^c,:}\|_{\textup{op}}$. Moreover, 
	\begin{align*}
		\|\wtilde{\bm{V}}_{I^c,:}\|_{\textup{op}} = \|\bm{V}^*_{I^c,:}\|_{\textup{op}} \le \|\bm{V}^*\|_{\textup{op}} \le 1, 
	\end{align*}
	where the first equality is because the rows of $\wtilde{\bm{V}}_{I^c,:}$ are exactly equal to the rows of $\bm{V}^*_{I^c,:}$ and the first inequality is because deleting rows cannot increase the operator norm, and the last inequality is because $\bm{V}^* \in \F$. Combining these observations we get that $\|\wtilde{\bm{V}}\|_{\text{op}} \le \|\wtilde{\bm{V}}_{I,:}\|_{\text{op}} + 1$, and so we focus on controlling the operator norm of $\wtilde{\bm{V}}_{I,:}$. We do that by applying a concentration inequality to the largest eigenvalue of the positive semi-definite (PSD) matrix $\wtilde{\bm{V}}_{I,:}^{\top} \wtilde{\bm{V}}_{I,:}$; the following is Theorem 1.1 of~\cite{tropp2012user} plus a simple estimate (see for example page 65 of~\cite{mitzenmacher2017probability}).
	\begin{theorem} \label{thm:matrixChernoff}
		Let $\bm{X}_1,\ldots, \bm{X}_n \in \R^{r \times r}$ be independent, random, symmetric matrices of size $r \times r$. Assume with probability 1 each $\bm{X}_i$ is PSD and has largest eigenvalue $\lambda_{\max}(\bm{X}_i) \le R$. Then
		\begin{align*}
			\Pr\left(\lambda_{\max}\left(\sum_{i = 1}^n \bm{X}_i\right) \ge \alpha \right) < r \cdot 2^{-\alpha/R}
		\end{align*} 
		for every $\alpha \ge 6 \cdot \lambda_{\max}\left(\E \left[\sum_{i = 1}^n \bm{X}_i \right] \right)$.
	\end{theorem}
	
	Notice that $\wtilde{\bm{V}}_{I,:}^{\top} \wtilde{\bm{V}}_{I,:}$ can be written as a sum of the independent PSD matrices:
	    \begin{align}
		    \wtilde{\bm{V}}_{I,:}^{\top} \wtilde{\bm{V}}_{I,:} = \sum_{i \in I} \wtilde{\bm{V}}_{i,:} \otimes \wtilde{\bm{V}}_{i,:} 
		     \label{eq:quadratic1}
	    \end{align}
	Therefore, to bound the operator norm $\|\wtilde{\bm{V}}_{I,:}\|_{\text{op}} = \sqrt{\lambda_{\max}(\wtilde{\bm{V}}_{I,:}^{\top} \wtilde{\bm{V}}_{I,:})}$ we will apply Theorem \ref{thm:matrixChernoff} over this sum. But for that, we will need to control the terms involved in the bound/assumptions of this theorem. 
	
	\begin{itemize} \itemsep6pt
	    \item \textbf{Term {$\lambda_{\max}(\wtilde{\bm{V}}_{i,:} \otimes \wtilde{\bm{V}}_{i,:})$.}} Notice that we have the identity $\wtilde{\bm{V}}_{i,:} = \frac{\e_i}{p_i} \bm{V}_{i,:}^*$. Therefore,
	    \begin{align*}
	        & ~ \lambda_{\max}\Big(\wtilde{\bm{V}}_{i,:} \otimes \wtilde{\bm{V}}_{i,:} \Big) = \lambda_{\max}\bigg(\bigg(\frac{\e_i}{p_i} \bm{V}_{i,:}^*\bigg) \otimes \bigg(\frac{\e_i}{p_i} \bm{V}_{i,:}^*\bigg)\bigg) \\
	        \le & ~ \lambda_{\max}\bigg(\frac{1}{p_i^2} (\bm{\bm{V}}_{i,:}^* \otimes \bm{V}_{i,:}^*)\bigg) = \frac{\| \bm{V}_{i,:}^*\|_2^2}{p_i^2} \le \frac{36 (\sum_{i' = 1}^d \|\bm{V}^*_{i',:} \|_2)^2}{k^2} \le 36,
	    \end{align*}
	    where the last inequality uses the fact $\bm{V}^* \in \CRo$ and hence $\sum_{i' = 1}^d \|\bm{V}^*_{i',:}\|_2 \le \sqrt{rk} \le k$. 

	    \item \textbf{Term $\lambda_{\max}\left(\E \left[\wtilde{\bm{V}}_{I,:}^{\top} \wtilde{\bm{V}}_{I,:}\right] \right)$.} First notice that we also have the identity $\wtilde{\bm{V}}_{i,:} = \e_i \,\frac{\bm{W}_{ii}}{p_i}\, \bm{T}_{i,:}$. Moreover, by definition of $p_i$ we have $\E \left(\e_i \frac{\bm{W}_{ii}}{p_i}\right)^2 \le \frac{6}{k}$, and so $\E\big[\wtilde{\bm{V}}_{i,:} \otimes \wtilde{\bm{V}}_{i,:}\big] \preceq \frac{6}{k}  (\bm{T}_{i,:} \otimes \bm{T}_{i,:})$. Therefore, using \eqref{eq:quadratic1} we have
	    \begin{align*}
		    \E \left[\wtilde{\bm{V}}_{I,:}^{\top} \wtilde{\bm{V}}_{I,:}\right] \,\preceq\, \sum_{i \in I} \frac{6}{k} (\bm{T}_{i,:} \otimes \bm{T}_{i,:}) \,\preceq\, \frac{6}{k} \sum_{i = 1}^d (\bm{T}_{i,:} \otimes \bm{T}_{i,:}) \,=\, \frac{6}{k}\, \bm{T}^{\top} \bm{T}. 
	    \end{align*}
	    By the guarantee provided by the Pietsch-Grothendieck factorization $\|\bm{T}\|_{\textup{op}} \le \sqrt{\pi/2}\, \|\bm{V}^*\|_{2 \rightarrow 1}$, and  since $\bm{V}^* \in \CRo$ we have $\|\bm{V}^*\|_{2 \rightarrow 1} \le \sqrt{k}$, so applying these bounds to the previous displayed inequality gives
	    \begin{align*}
		    \lambda_{\max}\left(\E \left[\wtilde{\bm{V}}_{I,:}^{\top} \wtilde{\bm{V}}_{I,:}\right] \right) \,\le\, \frac{6}{k} \lambda_{\max}(\bm{T}^\top \bm{T}) \,=\, \frac{6}{k} \|\bm{T}\|_{\textup{op}}^2 \,\le\, 3\pi. 
	    \end{align*}
	    
	\end{itemize}

	Having controlled these terms, we apply Theorem \ref{thm:matrixChernoff} with $\bm{X}_i = \wtilde{\bm{V}}_{i,:} \otimes \wtilde{\bm{V}}_{i,:}$, {$R=36$, and $\alpha = \max\{18 \pi, \, 36 \log 50 r\}$, which by the last item is at least $6 \cdot \lambda_{\max} (\E [\sum_{i \in I} \wtilde{\bm{V}}_{i,:} \otimes \wtilde{\bm{V}}_{i,:}])$, to get}
%
	%
	\begin{align*}
		\Pr \Big(\|\wtilde{\bm{V}}_{I,:}\|_{\text{op}} \ge \sqrt{\alpha} \Big) ~=~ \Pr \Big(\lambda_{\max}\left( \wtilde{\bm{V}}_{I,:}^{\top} \wtilde{\bm{V}}_{I,:} \right) \ge \alpha \Big) ~<~ r \cdot 2^{- \alpha / R} ~<~ \frac{1}{50}.  
	\end{align*}	
	Recalling we have $\|\wtilde{\bm{V}}\|_{\textup{op}} \le 1 + \|\wtilde{\bm{V}}_{I,:}\|_{\textup{op}}$, this gives that
	\begin{align}
	\|\wtilde{\bm{V}}\|_{\textup{op}} > 1 + \max\{\sqrt{18 \pi},\, 6 \sqrt{\log 50r}\}   \label{eq:PGnorm}
	\end{align}
	happens with probability at most $\frac{1}{50}$.
	

	\paragraph{Value.} We want to show that with good probability $\ip{\bm{C}}{\wtilde{\bm{V}}} \ge \frac{1}{2} \ip{\bm{C}}{\bm{V}^*}$. We use throughout the following observation: for each row $i$ we have $\ip{\bm{C}_{i,:}}{\bm{V}^*_{i,:}} \ge 0$, since the set $\CRo$ is symmetric with respect to flipping the sign of a row and $\bm{V}^*$ maximizes $\ip{\bm{C}}{\bm{V}^*} = \sum_{i = 1}^d \ip{\bm{C}_{i,:}}{\bm{V}^*_{i,:}}$. Since $\E \Big[\wtilde{\bm{V}}\Big]= \bm{V}^*$, we have $\E \left[\ip{\bm{C}_{I,:}}{\wtilde{\bm{V}}_{I,:}}\right] = \ip{\bm{C}_{I,:}}{\bm{V}_{I,:}^*}$. Again since $\wtilde{\bm{V}}_{i,:} = \frac{\e_i}{p_i} \bm{V}_{i,:}^*$, we have
	\begin{align*}
		\Var\left(\ip{\bm{C}_{I, :}}{\wtilde{\bm{V}}_{I,:}}\right) &\,=\,  \Var \bigg(\sum_{i \in I} \ip{\bm{C}_{i,:}}{\wtilde{\bm{V}}_{i, :}}\bigg) = \sum_{i \in I} \Var\Big(\tfrac{\e_i}{p_i} \ip{\bm{C}_{i,:}}{\bm{V}^*_{i,:}}\Big) \\
		&\le \sum_{i \in I} \frac{\ip{\bm{C}_{i,:}}{\bm{V}^*_{i,:}}^2}{p_i} \\
		&\,\le\, \frac{6 \sum_{i'} \|\bm{V}_{i',:}^*\|_2}{k} \cdot \sum_{i \in I} \frac{\ip{\bm{C}_{i,:}}{\bm{V}^*_{i,:}}^2}{\|\bm{V}^*_{i,:}\|_2} \, \\
& \le\, 6 \cdot \Big( \max_{i \in I} \Big\langle \bm{C}_{i,:}, \tfrac{\bm{V}^*_{i,:}}{\|\bm{V}^*_{i,:}\|_2} \Big\rangle\Big) \cdot \ip{\bm{C}_{I,:}}{\bm{V}^*_{I,:}},
	\end{align*}
	where the second inequality uses the definition of $p_i$ and the last inequality uses that $\sum_{i'} \|\bm{V}_{i', :}^*\|_2 \le \sqrt{rk} \le k$ (since $\bm{V}^* \in \CRo$). Moreover, since $\frac{\bm{V}^*_{i,:}}{\|\bm{V}^*_{i,:}\|_2}$ also belongs to $\CRo$ \footnote{Formally, we can append zero rows to $\frac{\bm{V}^*_i}{\|\bm{V}^*_i\|_2}$ so that the resulting matrix is in $\CRo$.}, the optimality of $\bm{V}^*$ guarantees that $\ip{\bm{C}_{i,:}}{\tfrac{\bm{V}_{i,:}^*}{\|\bm{V}^*_{i,:}\|_2}} \le \ip{\bm{C}}{\bm{V}^*}$, and so we have the variance upper bound
	\begin{align*}
		\Var \left(\ip{\bm{C}_{I,:}}{\wtilde{\bm{V}}_{I,:}} \right) \,\le\, 6\cdot \ip{\bm{C}}{\bm{V}^*}^2.
	\end{align*}
	Using the fact that $\ip{\bm{C}_{I^c,:}}{\wtilde{\bm{V}}_{I^c,:}} = \ip{\bm{C}_{I^c,:}}{\bm{V}^*_{I^c,:}}$ and the one-sided Chebychev inequality (Lemma \ref{lemma:cheby} in Appendix \ref{app:conc}) we get 
	\begin{align}
	& ~ \Pr\bigg(\ip{\bm{C}}{\wtilde{\bm{V}}} \le \frac{1}{2} \ip{\bm{C}}{\bm{V}^*}\bigg)  \nonumber \\
	= & ~ \Pr\bigg(\ip{\bm{C}_{I,:}}{\wtilde{\bm{V}}_{I, :}} + \ip{\bm{C}_{I^c,:}}{\wtilde{\bm{V}}_{I^c, :}} \le \frac{1}{2} \ip{\bm{C}}{\bm{V}^*}\bigg) \nonumber  \\ 
	= & ~ \Pr\bigg(\ip{\bm{C}_{I,:}}{\wtilde{\bm{V}}_{I,:}}
	\le -\ip{\bm{C}_{I^c,:}}{\bm{V}^*_{I^c,:}} + \frac{1}{2} \ip{\bm{C}}{\bm{V}^*}\bigg) \nonumber \\
	=&~ \Pr\bigg(\ip{\bm{C}_{I,:}}{\wtilde{\bm{V}}_{I,:}} \le \ip{\bm{C}_{I,:}}{\bm{V}^*_{I,:}} - \frac{1}{2} \ip{\bm{C}}{\bm{V}^*}\bigg) \nonumber  \\
	\leq & ~ \frac{6}{6 + 1/4} = \frac{24}{25}.\label{eq:PGvalue}
	\end{align}
	

	\paragraph{Concluding the proof of Lemma \ref{lemma:PGmain}.}	Taking a union bound over inequalities \eqref{eq:PGsparsity}, \eqref{eq:PGnorm}, and \eqref{eq:PGvalue}, we see that with positive probability $\wtilde{\bf{V}}$ satisfies all items from Lemma \ref{lemma:PGmain}. This shows the existence of the desired matrix $\bf{V}$ and concludes the proof.


\subsection{Semi-definite programming representable relaxation of $\mathcal{CR}1$} \label{sec:SDR-L21-norm} 

We now show how to obtain an SDP-representable relaxation of $\CRo$, which we denote by $\CRo'$, that still has essentially the same approximation guarantee.
	
For that, first one can {capture} the constraint $\|\bm{V}\|_{\textup{op}} \le 1$ of $\CRo$ by the constraint $\left[\begin{array}{cc} \bm{I}^d& \bm{V} \\ \bm{V}^T& \bm{I}^r  \end{array} \right] \succeq \bm{0}^{d+ r, d+ r}$, which is the convex hull of the \emph{Stiefel manifold} $\{\bm{V} \,|\, \bm{V}^{\top} \bm{V} = \bm{I}^r\}$~\cite{gallivan2010note}.

Next, we use the results from~\cite{tropp2009column} to obtain an SDP relaxation of the constraint $\|\bm{V}\|_{2 \rightarrow 1} \leq \sqrt{k}$. {More precisely, Theorem 3.1 of~\cite{tropp2009column} (in transpose form) states the following.}
	
	\begin{theorem}[Theorem 3.1 of~\cite{tropp2009column}] \label{thm:troppPG}
		Consider any matrix $\bm{V} \in \R^{d \times r}$. Suppose that $\bm{V} = \bm{W} \bm{T}$, where $\bm{W} \in \R^{d \times d}$ is a diagonal non-negative matrix with $\textup{Tr}(\bm{W}^2) = 1$. Then for every $\alpha \ge \|\bm{T}\|_{\textup{op}}$ we have 
	\begin{align}	
		\lambda_{\max}(\bm{V} \bm{V}^{\top} - \alpha^2 \bm{W}^2) \leq 0.  \label{eq:troppPG}
	\end{align}
	Conversely, if $\bm{W} \in \R^{d \times d}$ is a diagonal non-negative matrix with $\textup{Tr}(\bm{W}^2) = 1$ and $\alpha$ satisfies \eqref{eq:troppPG}, then there is a decomposition $\bm{V} = \bm{W} \bm{T}$ with $\|\bm{T}\|_{\textup{op}} \le \alpha$.
	\end{theorem}
	
Now consider a {matrix $\bm{V} \in \R^{d \times r}$ such that $\|\bm{V}\|_{2 \rightarrow 1} \leq \sqrt{k}$.} Let $\bm{V} = \bm{W} \bm{T}$ be its Pietsch-Grothendieck factorization, so $\bm{W}$ is a diagonal non-negative matrix with $\text{Tr}(\bm{W}^2) = 1$, and $\|\bm{T}\|_{\text{op}} \leq \sqrt{\frac{\pi}{2}}\|\bm{V}\|_{2 \rightarrow 1} \le \sqrt{\frac{\pi}{2}} k$. 

By Theorem \ref{thm:troppPG} we have that $\frac{\pi}{2} k \bm{W}^2 - \bm{V} \bm{V}^{\top} \succeq \bm{0}^{d,  d}$. This means that there is a vector $\bm{h} \in \R^d_+$, namely that given via $\textrm{diag}(\bm{h}) = \frac{\pi}{2} k \bm{W}^2$, that satisfies
\begin{eqnarray*}
\exists  \bm{h} \in \mathbb{R}^d_{+} ~:~ \sum_{i = 1}^d \bm{h}_i \leq  \frac{\pi}{2} k ~~\textrm{and}~~ \textup{diag}(\bm{h}) - \bm{V} \bm{V}^{\top} \succeq \bm{0}^{d,  d}.
\end{eqnarray*}
This is what we will use as the relaxation to the constraint $\|\bm{V}\|_{2 \rightarrow 1} \leq \sqrt{k}$.

 Putting the above observations into  SDP-representable constraints using Schur complement, together with the other constraints of $\mathcal{CR}1$, we then obtain our SDP-representable relaxation of $\CRo$:
\begin{align*}
    \mathcal{CR}1' := \left\{ \bm{V} \in \mathbb{R}^{d \times r} \left| 
    \begin{array}{rcl}
	\textup{there exists } \bm{h} \in \mathbb{R}^d_{+}, &\textup{s.t.}&\\
        \sum_{i = 1}^d \|\bm{V}_{i, :}\|_2 &\leq& \sqrt{rk} \\
        \sum_{i =1}^d \bm{h}_i &\leq& \frac{\pi}{2} k \\
       \left[\begin{array}{rl} \bm{I}^d& \bm{V} \\ \bm{V}^{\top}& \bm{I}^r \end{array}\right] &\succeq& \bm{0}^{d +r, d + r} \\
       \left[\begin{array}{rl} \textup{diag}(\bm{h})& \bm{V} \\ \bm{V}^{\top}& \bm{I}^r \end{array}\right] &\succeq& \bm{0}^{d +r, d + r} \\
    \end{array}
    \right.\right\}.
\end{align*}
%

It is clear from the above derivation that $\mathcal{CR}1 \subseteq \mathcal{CR}1'$. For the remainder of the section, we prove the approximation guarantee that $\mathcal{CR}1'$ provides for the convex hull of $\mathcal{F}$, namely that
\begin{align*}
       \mathcal{CR}1' \subseteq \rho_{\mathcal{CR}1'} \cdot \text{conv}(\mathcal{F}) 
\end{align*}
for $\rho_{\mathcal{CR}1'} = \sqrt{\frac{\pi}{2}} \cdot \rho_{\CRo}$, {where $\rho_{\CRo}$ is the approximation factor of $\CRo$ from Theorem \ref{thm:cr1}.} This will prove Theorem~\ref{thm:cr1'} stated in the introduction. 

\begin{proof}[of Theorem \ref{thm:cr1'}]
	{It suffices to show that $\CRo'$ is a $\sqrt{\frac{\pi}{2}}$-approximation of $\CRo$, namely $\CRo' \subseteq \sqrt{\frac{\pi}{2}} \cdot \CRo$.}
	
	For that, consider $\bm{V} \in \mathcal{CR}1'$ and its corresponding vector $\bm{h} \in \R^d_+$. We first show that $\|\bm{V}\|_{2 \rightarrow 1} \le \sqrt{\frac{\pi}{2} k}$ (which is a factor of $\sqrt{\frac{\pi}{2}}$ more than the constraint present in $\CRo$). Define the diagonal matrix $\bm{W}$ via $\bm{W}^2 = \frac{\textup{diag}(\bm{h})}{\sum_i \bm{h}_i}$, which then has $\textup{Tr}(\bm{W}^2) = 1$. Moreover, the constraints in $\CRo'$ guarantee that $$\bm{0}^{d,d} \,\preceq\, \textup{diag}(\bm{h}) - \bm{V} \bm{V}^\top \,=\, \bigg(\sum_{i=1}^d \bm{h}_i\bigg) \cdot  \bm{W}^2 - \bm{V} \bm{V}^\top.$$ Then using the ``conversely'' part of Theorem \ref{thm:troppPG} with $\alpha = \sqrt{\sum_i \bm{h}_i}$, there is $\bm{T}$ such that $\bm{V} = \bm{W} \bm{T}$ and $\|\bm{T}\|_{\textup{op}} \le \sqrt{\sum_i \bm{h}_i} \le \sqrt{\frac{\pi}{2} k}$, the last inequality also following from the constraints of $\CRo'$. Moreover, we can see that $\|\bm{V}\|_{2 \rightarrow 1} \le \|\bm{T}\|_{\textup{op}}$, since there are vectors $\bm{y}$ (with $\|\bm{y}\|_{\infty} \le 1$) and $\bm{x}$ (with $\|\bm{x}\|_2 \le 1$) such that $$\|\bm{V}\|_{2 \rightarrow 1} = \bm{y}^\top \bm{V} \bm{x} = \bm{y}^\top \bm{W} \bm{T} \bm{x} \le \|\bm{T} \bm{x}\|_2 \le \|\bm{T}\|_{\textup{op}},$$ where the first inequality is because $\|\bm{y}^\top \bm{W}\|_2 \le \sqrt{\textup{Tr}(\bm{W}^2)} = 1$. Together these observations give that $\|\bm{V}\|_{2 \rightarrow 1} \le \sqrt{\frac{\pi}{2} k}$ as desired. 
	
	Finally, directly from  the constraints of $\CRo'$ we have that $\|\bm{V}\|_{\textup{op}} \leq 1$ and $\sum_{i = 1}^d \|\bm{\bm{V}}_{i, :}\|_2 \leq \sqrt{rk}$. Therefore, these imply that the scaled matrix $\frac{\bm{V}}{\sqrt{\frac{\pi}{2}}}$ belongs to $\CRo$. Since this holds for every $\bm{V} \in \CRo'$, this shows that $\CRo' \subseteq \sqrt{\frac{\pi}{2}} \cdot \CRo$ as desired. 
\end{proof}


\subsection{Second order representable relaxation $\mathcal{CR}2$} \label{sec:convex-relax2}

The formulation $\mathcal{CR}1'$ only contains linear, second-order, and semi-definite constraints, which are easier to handle computationally compared to the original $\ell_{2 \rightarrow 1}$-norm constraint in $\mathcal{CR}1$. However, as we discussed in Section~\ref{introduction} and we see in Section~\ref{sec:numerical-exp}, the semi-definite constraints still make this relaxation $\mathcal{CR}1'$ difficult to scale for large instances in practice.

Therefore, for practical purposes we consider the following further relaxation involving only second-order cone constraints:

\begin{align}
	\mathcal{CR}2 := \left\{ \bm{V} \in \mathbb{R}^{d \times r} \,\left |\, 
	\begin{array}{llll}
		& \|\bm{V}_{:,j}\|_2^2 \leq 1 & \forall j \in [r] \\
		& \|\bm{V}_{:, j_1} \pm \bm{V}_{:, j_2} \|_2^2 \leq 2 & \forall j_1 \neq j_2 \in [r] \\
		& \|\bm{V}_{:, j}\|_1 \leq \sqrt{k} & \forall j \in [r] \\
		& \sum_{i = 1}^d \|\bm{V}_{i,:}\|_2 \leq \sqrt{rk}
	\end{array}\right.
	\right\}. \label{eq:CR2}
\end{align} 
This set is a relaxation of $\CRo$ obtained by considering
\begin{itemize}
    \item the constraint $\max_{\bm{x} : \|\bm{x}\|_2 \leq 1} \|\bm{Vx}\|_2 = \|\bm{V}\|_{\text{op}} \le 1$ only for the vectors $\bm{x} = \bm{e}^j$ and $\bm{x} = \frac{1}{\sqrt{2}} (\bm{e}^{j_1} \pm \bm{e}^{j_2})$, 
    \item the constraint $\max_{\|\bm{x}\|_2 \le 1} \|\bm{Vx}\|_1 = \|\bm{V}\|_{2 \rightarrow 1} \le \sqrt{k}$ only for the vectors $\bm{x} = \bm{e}^j$. 
\end{itemize}
In particular, this shows that $\CRt$ is a relaxation of $\CRo$ and hence a relaxation of $\F$. For the remainder of the section, we also prove the approximation guarantee that $\mathcal{CR}2$ provides for the convex hull of $\mathcal{F}$, i.e., 
\begin{align*}
	\conv(\mathcal{F}) \subseteq \mathcal{CR}2 \subseteq \rho_{\mathcal{CR}2} \cdot \conv(\mathcal{F})
\end{align*}
with $\rho_{\mathcal{CR}2} \leq 2\sqrt{r}$, which proves Theorem \ref{thm:cr2}. 

We would like to point out that most of the constraints in $\CRt$ are purely intended to tighten the relaxation. Only the constraints  $\|\bm{V}_{:,j}\|_2^2 \leq 1,  \forall j \in [r] $ and $\sum_{i = 1}^d \|\bm{V}_{i,:}\|_2 \leq \sqrt{rk}$ are required to obtain the guarantee that $\rho_{\mathcal{CR}2} \leq 2\sqrt{r}$.

\begin{proof}
Since we argued above that $\CRt$ is a relaxation of $\F$ it suffices to show the second inclusion $\CRt \subseteq (2\sqrt{r})\,\conv(\F)$. So consider any $\bm{V} \in \mathcal{CR}2$, and we will show $\bm{V} \in (2\sqrt{r}) \conv(\F)$.
\paragraph{Decomposition of $\bm{V}$.} Since the sets $\F$ and $\CRt$ are symmetric to row permutations, assume without loss of generality that the rows of $\bm{V}$ are sorted in non-decreasing length, namely $\|\bm{V}_{1,:}\|_2 \ge \|\bm{V}_{2,:}\|_2 \ge \cdots \geq \|\bm{V}_{d,:}\|_2$. Decompose $\bm{V}$ based on its top-$k$ largest rows, second top-$k$ largest rows, and so on, i.e., let $m = \lceil d / k \rceil$, $\bm{V} = \bm{V}^1 + \cdots+  \bm{V}^{m}$ with $\bm{V}^p \in \R^{d \times r}$ and 
\begin{align*}
	& ~ \text{supp}(\bm{V}^1) = \{1, \ldots, k\} =: I^1~, \,~ \ldots\,~,~ \\
	& ~ \text{supp}( \bm{V}^m) = \{d - (m - 1)k + 1, \ldots, d \} =: I^m,
\end{align*}
where $|I^1| = \cdots = |I^{m - 1}| = k$ and $|I^m| \leq k$. For each $p = 1, \ldots, m$, consider the normalized matrix $\bm{V}^p / \|\bm{V}^p\|_{\text{op}}$; we have $\|\bm{V}^p / \|\bm{V}^p\|_{\text{op}}\|_0 \leq k$ and $\|\bm{V}^p / \|\bm{V}^p\|_{\text{op}}\|_{\text{op}} = 1$, thus $\bm{V}^p / \|\bm{V}^p\|_{\text{op}} \in \text{conv}(\mathcal{F})$. Therefore, $\bm{V}$ can be decomposed as follows:
	\begin{align}
		\bm{V} = \bm{V}^1 + \cdots \bm{V}^{m} = & ~ \|\bm{V}^1\|_{\text{op}} \frac{\bm{V}^1}{\|\bm{V}^1\|_{\text{op}}} + \cdots + \|\bm{V}^m\|_{\text{op}} \frac{\bm{V}^m}{\|\bm{V}^m\|_{\text{op}}} \nonumber \\ 
		\Rightarrow ~~ \frac{\bm{V}}{\sum_{p = 1}^m \|\bm{V}^p\|_{\text{op}}} = & ~ \left(\frac{\|\bm{V}^1\|_{\text{op}}}{\sum_{p = 1}^m \|\bm{V}^p\|_{\text{op}}} \right) \frac{\bm{V}^1}{\|\bm{V}^1\|_{\text{op}}} + \cdots \label{eq:avgV} \\
		& ~ + \left(\frac{\|\bm{V}^m\|_{\text{op}}}{\sum_{p = 1}^m \|\bm{V}^p\|_{\text{op}}}\right) \frac{\bm{V}^m}{\|\bm{V}^m\|_{\text{op}}} \in \text{conv}(\mathcal{F}).     \notag
 	\end{align}
\paragraph{Controlling the normalization term $\sum_{p = 1}^m \|\bm{V}^p\|_{\text{op}}$.} 
Since the norm $\|\cdot\|_{\textup{op}}$ is always at most $\|\cdot\|_{F}$, we have
 	\begin{align*}
 		\|\bm{V}^p\|_{\text{op}} ~\stackrel{(\star)}{\le}~ \sqrt{ \sum_{i \in I^p} \|\bm{V}_{i,:}\|_2^2 }  ~\le~ \sqrt{\sum_{j = 1}^r \|\bm{V}_{:,j}\|_2^2 } \leq \sqrt{r},
 	\end{align*}
where the last inequality follows from the constraint $\|\bm{V}_{:,j}\|_2^2 \leq 1$ present in the description of $\mathcal{CR}2$.
 	
Furthermore, we can bound the $\ell_2$-norm of each of the rows of $\bm{V}^p$ by the average of the rows of $\bm{V}^{p-1}$, since the rows of $\bm{V}$ are sorted in non-decreasing length. Employing these bounds we get
 	\begin{align}
 		\sum_{p = 1}^m \|\bm{V}^p\|_{\text{op}} = & ~ \|\bm{V}^1\|_{\text{op}} +  \sum_{p = 2}^m \|\bm{V}^p\|_{\text{op}} \notag\\
 		\leq & ~ \|\bm{V}^1\|_{\text{op}} + \sum_{p = 2}^m \sqrt{ \sum_{i \in I^p} \|\bm{V}_{i,:}\|_2^2 } & \textrm{(by $(\star)$)}\notag \\
 		\leq & ~ \sqrt{r} + \sum_{p = 2}^m \sqrt{ \left( \frac{ \sum_{i \in I^{p - 1}} \|\bm{V}_{i,:}\|_2 }{k} \right)^2 \cdot k } \notag\\
 		= & ~ \sqrt{r} + \frac{1}{\sqrt{k}} \cdot \sum_{p = 2}^m \sum_{i \in I^{p - 1}} \|\bm{V}_{i,:}\|_2  \notag\\
 		\leq & ~ \sqrt{r} + \frac{1}{\sqrt{k}} \sum_{i = 1}^d \|\bm{V}_{i,:}\|_2   
 		\leq  ~ 2\sqrt{r},   \label{eq:sqrtR}
 	\end{align}
 	where the second inequality holds since
 	\begin{align*}
 	    \frac{ \sum_{i' \in I^{p - 1}} \|\bm{V}_{i',:} \|_2 }{k} ~ \geq ~ \|\bm{V}_{i,:} \|_2~, ~~ \forall ~ i \in I^p ~~ \forall ~ p \in \{2, \ldots, m\} 
 	\end{align*}
 	by the non-decreasing sorting, and the last inequality holds since the constraint $\sum_{i = 1}^d \|\bm{V}_{i,:}\|_2 \leq \sqrt{rk}$ is in the description of $\mathcal{CR}2$. Combining inequalities \eqref{eq:avgV} and \eqref{eq:sqrtR} we have
 	\begin{align*}
 		\bm{V} \in \left( \sum_{p = 1}^m \|\bm{V}^p\|_{\text{op}} \right) \cdot \text{conv}(\mathcal{F}) \subseteq (2\sqrt{r}) \cdot \text{conv}(\mathcal{F}).
 	\end{align*}
	concluding the proof of the theorem. 
\end{proof}
}


\section{Convex IP formulation for obtaining dual bounds for \ref{eq:Multi-SPCA}} \label{sec:dual-bound}

Based on the results in Section~\ref{sec:convex-relax}, we can set-up the following optimization problem:
\begin{align*}
	\textup{opt}^{\mathcal{CR}i} := \max_{\bm{V} \in \mathcal{CR}i} \text{Tr} \left( \bm{V}^{\top} \bm{A} \bm{V} \right).  \tag{CRi-Relax} \label{eq:SOC-relax}
\end{align*}
The following is a straightforward corollary of Theorem~\ref{thm:cr1}, Theorem~\ref{thm:cr1'} and Theorem~\ref{thm:cr2}: 
\begin{corollary}\label{coro:cr2}
$\textup{opt}^{\mathcal{F}} \leq \textup{opt}^{\mathcal{CR}i} \leq \rho_{\mathcal{CR}i}^2 \cdot \textup{opt}^{\mathcal{F}}$ for $i \in \{1, 1',2\}$.
\end{corollary}

The challenge of solving~\ref{eq:SOC-relax} is that the objective function is non-concave. Indeed, for the case $r = 1$, Corollary~\ref{coro:cr2} provides constant multiplicative approximation ratios to \ref{eq:Multi-SPCA}; thus, the  inapproximability results for \ref{eq:Multi-SPCA} with $r = 1$ from \cite{chan2016approximability,magdon2017np} imply that solving~\ref{eq:SOC-relax} to optimality is NP-hard. Therefore, we construct a further concave relaxation of the objective function.

\subsection{Piecewise linear upper approximation of objective function}\label{sec:plwconstruction}
Let $\bm{A} = \sum_{j = 1}^d \lambda_j \bm{a}_j \bm{a}_j^{\top}$ be the eigenvalue decomposition of sample covariance matrix $\bm{A}$ with $\lambda_1 \geq \cdots \geq \lambda_d \geq 0$. The objective function then can be represented as a summation 
\begin{align*}
	\text{Tr} \left( \bm{V}^{\top} \bm{A} \bm{V} \right) =  \sum_{j = 1}^d \lambda_j \sum_{i = 1}^r (\bm{a}_j^{\top} \bm{v}_i)^2,
\end{align*}where $\bm{v}_i ~ (= \bm{V}_{:,i}) \in \mathbb{R}^d$ denotes the $i$th column of $\bm{V}$ such that $\bm{V} = \left( \bm{v}_1 ~|~ \ldots ~|~ \bm{v}_r \right)$. Define the auxiliary variables $g_{ji} = \bm{a}_j^{\top} \bm{v}_i$ for $(j,i) \in [r] \times [d]$. Let $j_1,\ldots,j_d$ be the indices of the coordinates of $\bm{a}_j$ sorted in non-decreasing absolute value, namely 
\begin{align*}
	|[\bm{a}_j]_{j_1}| \geq \ldots \geq |[\bm{a}_j]_{j_k}| \geq \ldots \geq |[\bm{a}_j]_{j_d}|,
\end{align*}
and let 
\begin{align}
	\theta_{j} = \sqrt{ [\bm{a}_j]_{j_1}^2 + \cdots + [\bm{a}_j]_{j_k}^2} \label{eq:theta-j}
\end{align}
be the $\ell_2$-norm of the top-$k$ largest absolute entries of $\bm{a}_j$. Since $\bm{v}_i$ is supposed to be $k$-sparse with $2-$norm being at most $1$, it is easy to observe that $g_{ji}$ is within the interval $[-\theta_{j}, \theta_{j}]$. 

\paragraph{Piecewise linear approximation:} To relax the non-convex objective, we can upper approximate each quadratic term $g_{ji}^2$ by a piecewise linear function based on a new auxiliary variable $\xi_{ji}$ via \textit{special ordered sets type 2} (SOS-II)~\cite{wolsey1999integer} constraints (PLA) as follows,
\begin{align*}
	\text{PLA}([d] \times [r]) := \left\{ (g, \xi, \eta) \, \left |
	\begin{array}{llll}
		g_{ji} = \bm{a}_j^{\top} \bm{v}_i & (j,i) \in [d] \times [r] \\
		g_{ji} = \sum_{\ell = -N}^N \gamma_{ji}^{\ell} \eta_{ji}^{\ell} & (j,i) \in [d] \times [r] \\
		\xi_{ji} = \sum_{\ell = -N}^N \left( \gamma_{ji}^{\ell} \right)^2 \eta_{ji}^{\ell} & (j,i) \in [d] \times [r] \\
		\left( \eta_{ji}^{\ell} \right)_{\ell = -N}^N \in \text{SOS-II} & (j,i) \in [d] \times [r] \\
		|g_{ji}| \le \theta_j & (j,i) \in [d] \times [r] 
	\end{array} \right. ,
	\right\}
\end{align*}
where for each $(j,i) \in [d] \times [r]$, $\left( \eta_{ji}^{\ell} \right)_{\ell = -N}^N$ is the set of SOS-II variables, and $\left( \gamma_{ji}^{\ell} \right)_{\ell = -N}^N$ is the corresponding set of splitting points that satisfy: 
\begin{align*}
\begin{small}
	\underbrace{\gamma_{ji}^{-N}}_{= - \theta_j} \leq \cdots \leq \underbrace{\gamma_{ji}^{0}}_{= 0} \leq \cdot \cdots \leq \underbrace{\gamma_{ji}^{N}}_{= \theta_j}
\end{small}
\end{align*}
and split the region $[-\theta_{j}, \theta_{j}]$ into $2N$ equal intervals. See Figure~\ref{fig:SOS_II} for an example. 
\begin{figure}[ht]
\centering
\includegraphics[width=0.75\textwidth,trim={0 0.5cm 0 0},clip]{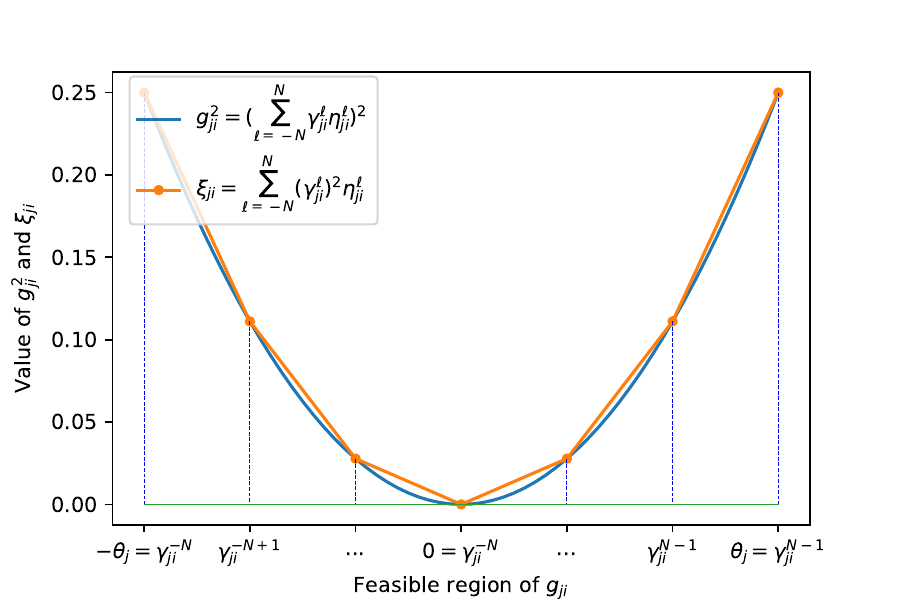}
\caption{The quadratic function $g_{ji}^2$ is upper approximated by a piecewise linear function $\xi_{ji}$ using SOS-II constraints for all $(j,i) \in [d] \times [r]$.}
\label{fig:SOS_II}

\end{figure} 
By using PLA, we  arrive at the following \emph{convex integer programming} problem,
\begin{align*}
	\begin{array}{rllll}
		\text{ub}^{\mathcal{CR}i} := \max & \sum_{j = 1}^d \lambda_j  \sum_{i = 1}^r \xi_{ji} \\
		\text{s.t.} & \bm{V} \in \mathcal{CR}i \\ 
		& (g, \xi, \eta) \in \text{PLA}([d] \times [r])
	\end{array} \tag{CIP} \label{eq:CIP} 
\end{align*}
where $\mathcal{CR}i$ is the convex set defined in Section~\ref{sec:convex-relax1} or Section~\ref{sec:convex-relax2} for $i \in \{1, 1', 2\}$ respectively, and $\text{PLA}$ is the set of constraints for piecewise-linear upper approximation of objective. Note that we say this is a convex integer program since SOS-II is modeled using binary variables. 


\subsection{Guarantees on the upper bounds from the convex integer program} \label{sec:ub-PLA}

Here we present the worst-case guarantee on the upper bound from solving convex integer program in the form of an affine function of $\text{opt}^{\mathcal{F}}$. The following theorem is a more precise restatement of Theorem \ref{thm:CIP} from the introduction. 

\begin{manualtheorem}{\ref{thm:CIP} (restated)}
For every positive integers $d, k, r, N$ such that $1 \leq r \leq k \leq d$, let $\bm{A} \in \mathbb{R}^{d \times d}$ be the sample covariance matrix. Let $\textup{opt}^{\mathcal{F}} = \textup{opt}^{\mathcal{F}}(\bm{A})$ be the optimal value of \ref{eq:Multi-SPCA}. Let $\textup{ub}^{\mathcal{CR}i} = \textup{ub}^{\mathcal{CR}i}(\bm{A})$ be the upper bound obtained from solving the \eqref{eq:CIP} using $\mathcal{CR}i$ convex relaxation for $i \in \{1, 1', 2\}$ with the $\textup{PLA}$ piecewise linear approximation set. Then
\begin{align*}
	\textup{opt}^{\mathcal{F}}(\bm{A}) \leq \textup{ub}^{\mathcal{CR}i} \leq \rho_{\mathcal{CR}i}^2 \cdot \textup{opt}^{\mathcal{F}}(\bm{A})+ \underbrace{\sum_{j = 1}^d \frac{r \theta_j^2 \lambda_j(\bm{A})}{4 N^2}}_{\emph{additive-term}, \textup{add}(\bm{A})}, \ \quad \textup{ for } i \in \{1, 1', 2\}, 
\end{align*}
where for every $j \in [d]$, $\lambda_j(\bm{A})$ is the $j$th eigenvalue of the sample covariance matrix $\bm{A}$, and $\theta_j$ is defined in \eqref{eq:theta-j}. 
\end{manualtheorem}

\begin{proof}
Based on the construction for \ref{eq:CIP}, the objective function $\text{Tr}\left( \bm{V}^{\top} \bm{A} \bm{V} \right)$ satisfies
\begin{align*}
	\sum_{j = 1}^d \lambda_j(\bm{A}) \sum_{i = 1}^r (\bm{a}_j^{\top} \bm{v}_i)^2 = \sum_{j = 1}^d \lambda_j(\bm{A}) \sum_{i = 1}^r g_{ji}^2. 
\end{align*} 
By Corollary~\ref{coro:cr2}, we have
\begin{align*}
	\max_{\bm{V} \in \mathcal{CR}i} \left( \bm{V}^{\top} \bm{A} \bm{V} \right) = & ~ \max_{\substack{\bm{V} \in \mathcal{CR}i}} \sum_{j = 1}^d \lambda_j(\bm{A}) \sum_{i = 1}^r g_{ji}^2 \leq \rho_{\mathcal{CR}i}^2 \cdot \text{opt}^{\mathcal{F}},
\end{align*}
for $i \in \{1,1',2\}$. Note that $g_{ji} \in [- \theta_j, \theta_j]$  and we have split the interval $[- \theta_j, \theta_j]$ evenly via splitting points $(\gamma_{ji}^{\ell})_{\ell = -N}^N$ such that $\gamma_{ji}^{\ell} = \frac{\ell}{N} \cdot \theta_j$. For a given $j \in [d]$ and $i \in [r]$, by the definition of SOS-II sets, let $g_{ij} =  \gamma_{ji}^{\ell^{\ast}} \eta_{j,i}^{\ell^{\ast}} + \gamma_{ji}^{\ell^{\ast} + 1} \eta_{j,i}^{\ell^{\ast} + 1} $, $\xi_{ji}  = (\gamma_{ji}^{\ell^{\ast}})^2 \eta_{j,i}^{\ell^{\ast}} + (\gamma_{ji}^{\ell^{\ast} + 1})^2 \eta_{j,i}^{\ell^{\ast} + 1}$ and $ \eta_{j,i}^{\ell^{\ast}} + \eta_{j,i}^{\ell^{\ast} + 1} = 1$ for some $\ell^{\ast} \in \{ -N, \dots, N -1\}$. Thus  we have:  
\begin{align*}
\xi_{ji} - g_{ji}^2 = & ~ \left( (\gamma_{ji}^{\ell^{\ast}})^2 \eta_{j,i}^{\ell^{\ast}} + (\gamma_{ji}^{\ell^{\ast} + 1})^2 \eta_{j,i}^{\ell^{\ast} + 1} \right) - \left( \gamma_{ji}^{\ell^{\ast}} \eta_{j,i}^{\ell^{\ast}} + \gamma_{ji}^{\ell^{\ast} + 1} \eta_{j,i}^{\ell^{\ast} + 1}  \right)^2 \\
= & ~ (\gamma_{ji}^{\ell^{\ast}})^2 \eta_{j,i}^{\ell^{\ast}} + (\gamma_{ji}^{\ell^{\ast} + 1})^2 \eta_{j,i}^{\ell^{\ast} + 1} - (\gamma_{ji}^{\ell^{\ast}})^2 (\eta_{j,i}^{\ell^{\ast}})^2 - (\gamma_{ji}^{\ell^{\ast} + 1})^2 (\eta_{j,i}^{\ell^{\ast} + 1})^2  \\
&- 2 \gamma_{ji}^{\ell^{\ast}} \eta_{j,i}^{\ell^{\ast}} \gamma_{ji}^{\ell^{\ast} + 1} \eta_{j,i}^{\ell^{\ast} + 1} \\
= & ~ \left( \gamma_{ji}^{\ell^{\ast} + 1} - \gamma_{ji}^{\ell^{\ast}} \right)^2 \eta_{ji}^{\ell^{\ast}} \eta_{ji}^{\ell^{\ast} + 1} 
=  ~ \frac{\theta_j^2}{N^2} \eta_{ji}^{\ell^{\ast}} \eta_{ji}^{\ell^{\ast} + 1} \leq \frac{\theta_j^2}{4 N^2}.
\end{align*}
Therefore, the objective function in \ref{eq:CIP} satisfies
\begin{align*} 
	\sum_{j = 1}^d \lambda_j(\bm{A}) \sum_{i = 1}^r \xi_{ji} \leq & ~ \sum_{j = 1}^d \lambda_j(\bm{A}) \sum_{i = 1}^r g_{ji}^2 + \sum_{j = 1}^d \frac{r \theta_j^2 \lambda_j(\bm{A})}{4 N^2} \\
	\leq & ~ \rho_{\mathcal{CR}i}^2 \cdot \text{opt}^{\mathcal{F}}(\bm{A}) + \sum_{j = 1}^d \frac{r \theta_j^2\lambda_j(\bm{A})}{4 N^2},
\end{align*}
which completes the proof. 
\end{proof}

Note that since $|\theta_j| \leq 1$ ($\theta_j$ is the two-norm of a sub-vector of a unit vector), we have that 
$$\underbrace{\sum_{j = 1}^d \frac{r \theta_j^2 \lambda_j(\bm{A})}{4 N^2}}_{\text{additive-term}, = \text{add}(\bm{A})} \leq \sum_{j = 1}^d \frac{r \lambda_j(\bm{A}) }{4 N^2} = \textup{Tr}(\bm{A})\cdot \frac{r}{4N^2}.$$

\section{{Greedy heuristic for} \ref{eq:Multi-SPCA}} \label{sec:primal-bound}
In order to evaluate the dual bounds produced by the convex integer program from the previous section, we also need good feasible solutions for \ref{eq:Multi-SPCA}. As mentioned in the introduction, we are not aware of any heuristics for the general case $r > 1$, so in this section, we describe the optimized version of the natural greedy heuristic that we will use. 	
	
We can view \ref{eq:Multi-SPCA} as the problem
\begin{align*}
\max_{S \subseteq [d], ~ |S|= k} ~~ & ~~ f(S) ,
\end{align*}
where
\begin{align*}
f(S):= \left( \max_{\bm{V} \in \mathbb{R}^{d\times r} \,|\, \bm{V}^{\top}\bm{V} = \bm{I}^r, ~ \textup{supp}(\bm{V}) = S} \text{Tr} \left( \bm{V}^{\top} \bm{A} \bm{V} \right) \right), \end{align*}
and hence solving \ref{eq:Multi-SPCA} reduces to selecting the correct support set $S$. Thus, a natural algorithm is the \emph{1-neighborhood} local search that starts with a support set $S$ and removes/adds one index to improve the value $f(S)$.\footnote{Another idea we explore is to find a principal submatrix whose determinant is near-maximal using a greedy algorithm, see Algorithm~\ref{algo:GH-det} in Appendix~\ref{app:PH-det}. More on this in Section \ref{sec:num-lb-method}.} The main issue with this strategy is that it requires an expensive eigendecomposition computation for each candidate pair $i$/$j$ of indices to be removed/added to evaluate the function $f$. Here we propose a much more efficient strategy that solves a proxy version of this local search move that requires only one eigen-decomposition per round. 

For that we rewrite the problem as follows. Given a sample covariance matrix $\bm{A}$, let $\bm{A}^{1/2}$ be its positive semi-definite square root such that $\bm{A} = \bm{A}^{1/2} \bm{A}^{1/2}$. Observe that $ \|\bm{A}^{\frac{1}{2}} - \bm{V}\bm{V}^{\top}\bm{A}^{\frac{1}{2}}\|_F^2 = \text{Tr}(\bm{A}) - \text{Tr}(\bm{V}^{\top}\bm{A}\bm{V}),$
and therefore we may equivalently solve the following problem:
\begin{align*}
	\begin{array}{rlllll}
		\min_{\bm{V} \in \mathbb{R}^{d \times r}} & \left\| \bm{A}^{1/2} - \bm{V} \bm{V}^{\top} \bm{A}^{1/2} \right\|_F^2 & \text{s.t.} & \bm{V}^{\top} \bm{V} = \bm{I}^r, ~ \|\bm{V}\|_0 \leq k. 
	\end{array} \tag{SPCA-alt} \label{eq:SPCA-lasso}
\end{align*}	
	Moreover, \ref{eq:SPCA-lasso} can be reformulated into a \textit{two-stage (inner \& outer) optimization problem}:
\begin{align*}
\begin{array}{rrllll}
	\min_{S \subseteq [d], ~ |S| \leq k} & \min_{\bm{V}_S} & ~ \bar{f}(S, \bm{V}_S) & \text{s.t. } & ~ \bm{V}_S^{\top}\bm{V}_S  = \bm{I}^r
\end{array}
\end{align*}
where
\begin{eqnarray}\label{eq:fdefn}
\bar{f}(S, \black{\bm{M}}) := \| (\bm{A}^{1/2})_S - \black{\bm{M} \bm{M}^{\top}} (\bm{A}^{1/2})_S \|_F^2 + \|(\bm{A}^{1/2})_{S^C} \|_F^2
\end{eqnarray}
and $S^C := [d] \setminus S$.

In order to find a solution with small $\bar{f}(S, \bm{V}_S)$ again we use a greedy swap heuristic that removes/adds one index to $S$. However, we avoid eigenvalue computations by keeping $\M = \V_S$ fixed and finding an improved set $S'$ (i.e., with $\bar{f}(S',\M) \le \bar{f}(S,\M)$), and only then updating the term $\M$; only the second only needs 1 eigendecomposition of $\bm{A}_{S_t,S_t}$. We describe this in more detail, letting $S_t$ and $\V^t_{S_t}$ be the iterates at round $t$.
		
\paragraph{Leaving Candidate:}  In the $t$-th iteration, given the iterates $S_{t-1}$ and $\bm{V}^{t-1}_{S_{t-1}}$ from the previous iteration, for each index $j \in S_{t - 1}$, let $\Delta^{\oout}_j$ be 
\begin{align*}
\Delta^{\oout}_j := \| \bm{A}^{1/2}_j\|_2^2 - \left\| \bm{A}^{1/2}_{S_{t - 1}}  - \bm{V}_{S_{t - 1}} \bm{V}_{S_{t - 1}}^{\top} \bm{A}^{1/2}_{S_{t - 1}} \right\|_F^2.
\end{align*}
Then let $j^{\text{out}} := \argmin_{j \in S_{t - 1}} \Delta^{\oout}_j$ be the candidate to leave the set $S_{t - 1}$. 

\paragraph{Entering Candidate:} Similarly, for each $j \in S_{t - 1}^C$ define $\Delta^{\iin}_j$ as 
\begin{align*}
	\Delta^{\iin}_j := \| \bm{A}^{1/2}_j\|_2^2 - \left\| (\bm{A}^{1/2})_{S_{t - 1}^j} -  \bm{V}_{S_{t - 1}} \bm{V}_{S_{t - 1}}^{\top} (\bm{A}^{1/2})_{S_{t - 1}^j} \right\|_F^2,
\end{align*}
where $S_{t - 1}^j := S_{t - 1} - \{j^{\text{out}}\} + \{j\}$. 
Then let $j^{\text{in}} := \argmax_{j \in S_{t - 1}^C} \Delta^{\iin}_j$.

\paragraph{Update Rule:} If $\Delta^{\oout}_{j^{\oout}} <  \Delta^{\iin}_{j^{\iin}}$ we perform the exchange with the candidates above, namely set $S_t = S_{t-1} - \{j^{\oout}\} + \{j^{\iin}\}$. In addition, we set $\bm{V}^t_{S_t}$ to be the minimizer of $\min\{f(S_t, \bm{M}) : \bm{M}^\top \bm{M} = \bm{I}^r\}$; for that we compute the eigendecomposition $\bm{A}_{S_t,S_t} = \bm{U}_{S_t} \bm{\Lambda}_{S_t} \bm{U}_{S_t}^\top$ of $\bm{A}_{S_t,S_t}$ and set $\bm{V}^t_{S_t} = (\bm{U}_{S_t})_{\star, [r]}$ to be the eigenvectors corresponding to top $r$ eigenvalues. 

If $\Delta^{\oout}_{j^{\oout}} \ge \Delta^{\iin}_{j^{\iin}}$ the algorithm stops and return the matrix $\bm{V}$ where in rows $S_{t-1}$ equals $\bm{V}^{t-1}_{S_{t-1}}$ (i.e., $\bm{V}_{S_{t-1}} = \bm{V}^{t-1}_{S_{t-1}}$) and in rows $S^C_{t-1}$ equals zero. The complete pseudocode is presented {in Appendix \ref{app:pseudo}}. 

We observe that even though our procedure works only with a proxy of the original function $f$ of the natural greedy heuristic, by construction it still finds support sets $S$ that monotonically decrease this objective function {(see Appendix \ref{app:monoProof} for a proof)}. 

\begin{lemma} \label{lemma:primal}
Algorithm~\ref{algo:LSM} is a monotonically decreasing algorithm with respect to the objective function {$f$, namely $f(S_t) < f(S_{t-1})$ for every iteration $t$.} 
\end{lemma}


\section{Computational experiments} \label{sec:numerical-exp}


In this section, we conduct computational experiments on fairly large instances to illustrate the efficiency of our proposed methods and assess their quality in finding good primal solutions and in proving good dual bounds. 

\subsection{Methods for comparison}

\subsubsection{Methods for dual bounds} 

	
In order to generate dual bounds, we implemented a version of our convex integer programming formulation~\eqref{eq:CIP}. Moreover, we add several enhancements to the proposed \eqref{eq:CIP} like reduction of the number of SOS-II constraints and cutting planes in order to improve its efficiency (see~\cite{dey2018convex} for related ideas for the case of $r =1$). This implemented version is called \ref{eq:CIP-impl}, and is described in detail in  Appendix~\ref{app:reduce-running-time}. For all experiments we use $N =40$ as the level of discretization for the objective function in \ref{eq:CIP-impl}. (For large instances, we additionally use a dimension reduction technique, which we discuss later.)

We compare our proposed dual bound with the following two baselines:
\begin{itemize} \itemsep6pt
	\item \textbf{Baseline 1:} Sum of the diagonal entries of the ``best'' sub-matrix:
		\begin{align*}
			\text{Baseline1} := & \bm{A}_{j_1, j_1} + \cdots + \bm{A}_{j_k, j_k},
		\end{align*}
where $j_1, \ldots, j_d$ is the permutation of the indices that makes the diagonal of $\bm{A}$ sorted in non-increasing order, namely $\bm{A}_{j_1, j_1} \ge \bm{A}_{j_2, j_2} \geq \cdots \ge \bm{A}_{j_d, j_d}$. Note that the sum of $\bm{A}_{j_1, j_1}, \ldots, \bm{A}_{j_k, j_k}$ is equal to the sum of the eigenvalues of the sub-matrix indexed by $\{j_1, \ldots, j_k\}$ in $\bm{A}$, then Baseline-1 can be viewed as an upper bound for the optimal value of \ref{eq:Multi-SPCA}. Moreover, Baseline-1 is tight when we have $r = k$.
		
	
	\item \textbf{Baseline 2:} To obtain a semi-definite programming relaxation, we go to the lifted space where we define variable $\bm{P}:= \bm{V} \bm{V}^{\top}$.  Note that it is easy to verify that if $\bm{V} \in \mathcal{F}$, then $\|\bm{P}\|_{1} \leq rk, ~ \text{Tr}(\bm{P}) = r$. Moreover, all the constraints defining $\mathcal{CR}1'$ can be naturally written in lifted space except for the constraints $\sum_{i=1}^d\| \bm{V}_{i,:}\|_2 \leq \sqrt{rk} $. Thus, we obtain the following semi-definite programming relaxation: 
	\begin{align*}
        \begin{array}{rlll}
            \text{SDP} := \max_{\bm{P}, \bm{h}} & \text{Tr}(\bm{P} \bm{A})  \\
            \text{s.t.} & \|\bm{P}\|_{1} \leq rk, ~ \text{Tr}(\bm{P}) = r, ~ \bm{I}^{d} \succeq \bm{P} \succeq \bm{0}^{d, d} \\
            & ~ \sum_{i =1}^d \bm{h}_i \leq \frac{\pi}{2} k, \\
            & ~  ~ \textup{diag}(\bm{h}) - \bm{P} \succeq \bm{0}^{d, d},
        \end{array},  
    \end{align*} 
    which outputs a baseline upper bound for the \ref{eq:Multi-SPCA} problem.
\end{itemize}

\subsubsection{Parameters for the primal algorithm (lower bounds)} \label{sec:num-lb-method}

To obtain good feasible solutions, we implemented the modified greedy neighborhood search (Algorithm~\ref{algo:LSM}) proposed in Section~\ref{sec:primal-bound}. For each instance, we run this algorithm $400$ times, where each time, we pick the initial support set $S_0$ as a uniformly random subset of $[d]$ of size $k$. We allow a maximum of $d$ iterations. The objective function value corresponding to the best solution from the 400 runs is declared as the lower bound. 
	
We have also compared this modified greedy neighborhood search (Algorithm~\ref{algo:LSM}) with a greedy algorithm that tries to maximize the determinant of the $k \times k$ submatrix. The details of this algorithm (Algorithm~\ref{algo:GH-det}) is presented in Appendix \ref{app:PH-det}. Based on the numerical results reported in Table~\ref{tab:artificial-instances-primal} and Table~\ref{tab:real-instances-primal}, the proposed greedy neighborhood search Algorithm~\ref{algo:LSM} outperforms the Algorithm~\ref{algo:GH-det} in all instances. 

\subsection{Instances for numerical experiments} \label{sec:instances}

We conducted numerical experiments on two types of instances. 

\subsubsection{Artificial instances}

These instances were generated artificially using ideas similar to that of the \textit{spiked covariance matrix}~\cite{deshpande2016sparse} that have been used often to test algorithms in the $r=1$ case. An instance \textbf{Artificial-$k^A$} is generated as follows.

We first choose a sparsity parameter $k^{\text{A}} \le \frac{d}{2}$ (which will be in the range $[30]$) and the  orthonormal vectors $\bm{u}_1$ and $\bm{u}_2$ of dimension $k^A$ given by 
\begin{align*}
	& ~ \bm{u}_1^{\top} = \bigg(  \frac{1}{\sqrt{k^{\text{A}}}}, \ldots, \frac{1}{\sqrt{k^{\text{A}}}} \bigg), & ~ \bm{u}_2^{\top} = \bigg( \frac{1}{\sqrt{k^{\text{A}}}}, - \frac{1}{\sqrt{k^{\text{A}}}}, \ldots, \frac{1}{\sqrt{k^{\text{A}}}}, - \frac{1}{\sqrt{k^{\text{A}}}}, \bigg).
\end{align*}
	The \emph{block spiked covariance matrix} $\bm{\Sigma} \in \R^{d \times d}$ is then computed as 
	\begin{align*}
	\bm{\Sigma} := \bm{\Sigma}_1 \oplus \bm{\Sigma}_2 \oplus \bm{I}^{d - 2k^{\text{A}}},
\end{align*} 
where 
$	\bm{\Sigma}_1 :=  ~ 55 \bm{u}_1 \bm{u}_1^{\top} + 52  \bm{u}_2 \bm{u}_2^{\top} \in \mathbb{R}^{k^{\text{A}} \times k^{\text{A}}}, \bm{\Sigma}_2 :=  ~ 50 \bm{I}_{k^{\text{A}}} \in \mathbb{R}^{k^{\text{A}} \times k^{\text{A}}}. $
Finally, we sample $M$ i.i.d. random vectors $\bm{x}_1, \ldots, \bm{x}_M \sim N(\bm{0}_d, \bm{\Sigma})$ from the normal distribution with covariance matrix $\bm{\Sigma}$ and create the instance $\bm{A}$ as the sample covariance matrix of these vectors:
\begin{align*}
	\bm{A} := \frac{1}{M} \left( \bm{x}_1 \bm{x}_1^{\top} + \cdots + \bm{x}_M \bm{x}_M^{\top}\right).
\end{align*}

In our experiments we used $d = 500$ (thus generating $500 \times 500$ matrices) and $M = 3000$ samples. Our experiments will focus on the cases $r=2$ and $r=3$ and we note that in these instances the optimal support set with cardinality $k^{\text{A}}$ is different for both choices of $r$.

\subsubsection{Real instances}
The second type of instances are four real instances using the colon cancer dataset (CovColon) from~\cite{alon1999broad}, the lymphoma dataset (Lymph) from~\cite{alizadeh2000distinct}, and Reddit instances Reddit1500 and Reddit2000 from~\cite{dey2018convex}. Table \ref{tab:real-instances} presents the size of each instance.

\begin{table}[ht]
    \centering
    \begin{tabular}{cccccccccc} 
    \toprule
        name & CovColon & Lymph & Reddit1500 & Reddit2000 \\ \midrule
        size & $500\times 500$ & $500\times 500$ & $1500 \times 1500$ & $2000 \times 2000$ \\
    \bottomrule
    \end{tabular} 
    \caption{Real instances}
    \label{tab:real-instances}
\end{table}
 
\subsection{Software \& hardware}
All numerical experiments are implemented on MacBookPro13 with a 2GHz Intel Core i5 CPU and 8GB 1867MHz LPDDR3 Memory. The (\ref{eq:CIP-impl}) model was solved using Gurobi 7.0.2. The Baseline-2 SDP relaxation was solved using Mosek version 9.1 with CVX in Matlab R2021a. 

\subsection{Performance measure} 
We measure the performances of \ref{eq:CIP-impl} and the baselines based on the primal-dual gap, defined as 
\begin{align*}
	\text{Gap} := \frac{\text{ub} - \text{lb}}{\text{lb}}. 
\end{align*}
Here $\text{ub} \in \{\text{ub}^{\text{impl}} ~ (\text{ub}^{\text{sub-mat}} \text{ in Section~\ref{sec:num-sub-matrix}}), \text{Baseline-1}, \text{Baseline-2}\}$ denotes the dual bound obtained from \ref{eq:CIP-impl} or baselines. The term $\text{lb}$ denotes the primal bound from the primal heuristic. 

\subsection{Numerical results for smaller instances} \label{sec:num-small-instance}
First, we perform experiments on smaller instances of size $100 \times 100$. These instances were constructed by picking the submatrix corresponding to the top 100 largest diagonal entries from each instance listed in Section~\ref{sec:instances}. We append a ``prime'' in the name of the instances to denote these smaller instances, e.g., Artificial-$k^A$' and CovColon'.

\paragraph{Time limits.} We set the time limit for \ref{eq:CIP-impl} to $60$ seconds and imposed no time limit on SDP. (We note that SDP terminated within 600 seconds on these smaller instances.) We also did not impose a time limit on the primal heuristic, and just noted that it took less than 120 seconds on all smaller instances.

The gaps obtained by the dual bounds using \ref{eq:CIP-impl}, Baseline1, SDP (Baseline 2), on these instances are presented in  Tables~\ref{tab:small-artificial-instances-results} and~\ref{tab:small-real-instances-results}. 

\begin{table}[ht]
    \centering
    \begin{tabular}{|c|c||c|c|c||c|c|c|c|c|c|c|} \nthickrule
name   & param $(r, k)$: & $(2, 10)$ & $(2, 20)$ & $(2, 30)$ &   $(3, 10)$ & $(3, 20)$ & $(3, 30)$   \\ \nthickrule
        Artificial-10' & CIP-impl & 0.031 & \bf 0.0004 & \bf 0.0003 & 0.04 & \bf 0.0005 & \bf 0.0004  \\
        $100 \times 100$& Baseline1 & 3.523 & 4.309 & 4.403 &   2.108 & 2.625 & 2.689  \\
        & SDP & \bf 0.029 & 0.0005 & \bf 0.0003 & \bf 0.03 & \bf 0.0005 & \bf 0.0004 \\
         \specialrule{.1em}{.05em}{.05em} 
        
        Artificial-20' & CIP-impl & 0.027 & \bf 0.011 & \bf 0.007 &   \bf 0.026 & \bf 0.011 & \bf 0.006  \\
        $100 \times 100$& Baseline1 & 3.58 & 7.838 & 8.251 &   2.094 & 4.942 & 5.216   \\
        & SDP & \bf 0.02 & 0.014 & 0.008 & 0.027 & 0.012 & \bf 0.006 \\
         \specialrule{.1em}{.05em}{.05em}

        Artificial-30' & CIP-impl & 0.071 & 0.022 & \bf 0.015 &   0.074 & 0.023 & \bf 0.012   \\
        $100 \times 100$& Baseline1 & 3.503 & 7.614 & 11.68 &   2.066 & 4.814 & 7.508  \\
        & SDP & \bf 0.037 & \bf 0.021 & 0.019 & \bf 0.046 & \bf 0.022 & 0.014 \\
         \specialrule{.1em}{.05em}{.05em}     
    \end{tabular}
    \caption{Gap values for smaller artificial instances with size $100 \times 100$}
    \label{tab:small-artificial-instances-results}
\end{table}

\begin{table}[ht]
    \centering
    \begin{tabular}{|c|c||c|c|c||c|c|c|c|c|c|c|} 
				\nthickrule
       	name  & param $(r, k)$: & $(2, 10)$ & $(2, 20)$ & $(2, 30)$ &   $(3, 10)$ & $(3, 20)$ & $(3, 30)$  \\ \nthickrule
         
        CovColon' & CIP-impl & 0.12 & 0.119 & \bf 0.094 &  0.127  & 0.124 & 0.104  \\
        $100 \times 100$& Baseline1 & \bf 0.063 & \bf 0.117 & 0.132  & \bf 0.052 & \bf 0.086 & \bf 0.098  \\
        & SDP & 0.428 & 0.450 & 0.442 & 0.434 & 0.452 & 0.438 \\
        \nthickrule
        
        Lymp' & CIP-impl & 0.329  & \bf 0.272 & \bf 0.269 &  0.225 & 0.296 & 0.32  \\
        $100 \times 100$& Baseline1 & \bf 0.095  & 0.277  & 0.392 &  \bf 0.049 & \bf 0.178 & \bf 0.297  \\
        & SDP & 0.355 & 0.324 & 0.31 & 0.390 & 0.340 & 0.352 \\
        \nthickrule
        
        Reddit1500' & CIP-impl & \bf 0.155 & \bf 0.139 & \bf 0.126 &  \bf 0.129 & \bf 0.109 & \bf 0.025  \\
        $100 \times 100$& Baseline1 & 0.695 & 0.396 & 0.99 &  1.197 & 0.811 & 1.294  \\
        & SDP & 0.205 & 0.216 & 0.176 & 0.158 & 0.188 & 0.177 \\
        \nthickrule
        
        Reddit2000' & CIP-impl & \bf 0.029 & \bf 0.014 & \bf 0.011 & \bf 0.092 & \bf 0.054 & \bf 0.011 \\
        $100 \times 100$& Baseline1 & 0.876 &1.426 & 1.794 &  0.638 & 1.075 & 1.333  \\
        & SDP & 0.097 & 0.061 & 0.029 & 0.093 & 0.064 & 0.031 \\
        \nthickrule  
    \end{tabular}
    \caption{Gap values for smaller real instances with size $100 \times 100$}
    \label{tab:small-real-instances-results}
\end{table}

\paragraph{Observations:}
\begin{itemize}
	\item In Table~\ref{tab:small-artificial-instances-results} we see that for the relatively easy artificial instances both \ref{eq:CIP-impl} and SDP find quite tight upper bounds. 
	\item In Table~\ref{tab:small-real-instances-results} we see that for real instances SDP is substantially dominated by both \ref{eq:CIP-impl} and Baseline1. 
\end{itemize}

Overall, on the 42 instances, 
\begin{itemize}
    \item the dual bounds from \ref{eq:CIP-impl} are the best for $26$ instances,
    \item the dual bounds from Baseline-1 are the best for $9$ instances,
    \item the dual bounds from Baseline-2 (SDP) are the best for $11$ instances.
\end{itemize}
Since the computation of Baseline-1 scales trivially in comparison to solving the SDP, and since SDP seems to produce dual bounds of poorer quality for the more difficult real instances, in the next section {we discarded SDP from the comparison}.


\subsection{Larger instances}
\subsubsection{Sub-matrix technique for larger instances} \label{sec:num-sub-matrix}

In order to scale the convex integer program \ref{eq:CIP-impl} to handle the larger matrices that are now up to $2000 \times 2000$, we employ the following ``sub-matrix technique'' to reduce the dimension. 

Given a \emph{sub-matrix ratio parameter} $m \geq 1$ satisfying $\lceil m k \rceil \le d$, let $S := \{j_1, \ldots, j_{\lceil m k \rceil}\}$, where $\bm{A}_{j_1, j_1} \geq \cdots \geq \bm{A}_{j_{\lceil m k \rceil}, j_{\lceil m k \rceil}}$, be the index set of the top-$\lceil m k \rceil$ largest diagonal entries of $\bm{A}$. Consider the blocked representation of the sample covariance matrix $\bm{A}$: 
\begin{align*}
	\bm{A} = \begin{pmatrix}
		\bm{A}_{S, S} & \bm{A}_{S, S^C} \\
		\bm{A}_{S, S^C}^{\top} & \bm{A}_{S^C, S^C} \\
	\end{pmatrix},
\end{align*}
where $S^C := [d] \setminus S$. Then the optimal value $\text{opt}^{\mathcal{F}}$ satisfies
\begin{align*}
	\text{opt}^{\mathcal{F}} = \max_{\bm{V} \in \mathcal{F}} & ~  \text{Tr}(\bm{V}^{\top} \bm{A} \bm{V}) \\
	= \max_{\bm{V} \in \mathcal{F}} & ~  \textup{Tr}\left( (\bm{V}_{S})^{\top} \bm{A}_{S,S} \bm{V}_{S} \right) + 2\, \text{Tr}\left( (\bm{V}_{S})^{\top} \bm{A}_{S, S^C} \bm{V}_{S^C} \right) \\
	& ~ + \textup{Tr}\left( (\bm{V}_{S^C})^{\top} \bm{A}_{S^C, S^C} \bm{V}_{S^C} \right). \tag{submatrix-tech} \label{eq:submatrix-tech}
\end{align*}
The first and third term have straight forward upper bounds. Now we need to consider the problem of finding an upper bound on $\textup{Tr}\left((\bm{V}_{S})^{\top} \bm{A}_{S, S^C} \bm{V}_{S^C} \right)$. 

Let $S^{\ast}$ be the global optimal row-support set of \ref{eq:Multi-SPCA}. Then 
\begin{align*}
	& ~ \textup{Tr}\left( (\bm{V}_{S})^{\top} \bm{A}_{S, S^C} \bm{V}_{S^C} \right) \\
	= & ~ \text{Tr} \left( 
	\begin{pmatrix}
		(\bm{V}_{S \cap S^{\ast}})^{\top} & (\bm{V}_{S \backslash S^{\ast}})^{\top} 
	\end{pmatrix} 
	\begin{pmatrix}
		\bm{A}_{S \cap S^{\ast}, S^C \cap S^{\ast}} & \bm{A}_{S \cap S^{\ast}, S^C \backslash       S^{\ast}} \\
		\bm{A}_{S \backslash S^{\ast}, S^C \cap S^{\ast}} & \bm{A}_{S \backslash S^{\ast}, S^C \backslash       S^{\ast}}
	\end{pmatrix}
	\begin{pmatrix}
		\bm{V}_{S^C \cap S^{\ast}}  \\
		\bm{V}_{S^C \backslash S^{\ast}}
	\end{pmatrix}
	\right) \\
	= & ~ \textup{Tr} \left( (\bm{V}_{S \cap S^{\ast}})^{\top}  \bm{A}_{S \cap S^{\ast}, S^C \cap S^{\ast}} \bm{V}_{S^C \cap S^{\ast}} \right).
\end{align*}
Since $\bm{V}^{\top} \bm{V} = \bm{I}^r$, then we have $\bm{V}_{S \cap S^{\ast}}^{\top} \bm{V}_{S \cap S^{\ast}} + \bm{V}_{S^C \cap S^{\ast}}^{\top} \bm{V}_{S^C \cap S^{\ast}} = \bm{I}^r$. Thus it is sufficient to consider the following optimization problem:
\begin{align*}
	2\max_{\bm{V}^1, \bm{V}^2} & ~ \textup{Tr}\left( (\bm{V}^1)^{\top} \bm{A}_{S \cap S^{\ast}, S^C \cap S^{\ast}} \bm{V}^2 \right) \text{ s.t. } (\bm{V}^1)^{\top} \bm{V}^1 + (\bm{V}^2)^{\top} \bm{V}^2 = \bm{I}^r,
\end{align*}
We show in Proposition~\ref{prop:submat}, proved in the appendix, that the above term is upper bounded by $\sqrt{r} \cdot \|\bm{A}_{(S \cap S^{\ast}), (S^C \cap S^{\ast}) }\|_F$.

Therefore, letting $\tilde{k} := |S \cap S^{\ast}|$ be the cardinality of the intersection, we can upper bound the right-hand side of \eqref{eq:submatrix-tech} as
\begin{align*}
	\text{opt}^{\mathcal{F}} \leq \text{ub}^{\text{CIP}}(\bm{A}_{S, S};\tilde{k}) + \sqrt{r}\cdot \|\bm{A}_{S \cap S^{\ast}, S^C \cap S^{\ast} }\|_F + \text{Baseline-1}(\bm{A}_{S^C, S^C};k - \tilde{k}),
\end{align*} 
where the first term $\text{ub}^{\text{CIP}}(\bm{A}_{S, S};\tilde{k})$ is the optimal value obtained from \ref{eq:CIP-impl} with covariance matrix $\bm{A}_{S, S}$ and sparsity parameter $\tilde{k}$ (if $\tilde{k} < r$, then reset $\tilde{k} = r$), and the third term is the value of Baseline-1 obtained from $\bm{A}_{S^C, S^C}$ with sparsity parameter $k - \tilde{k}$. 

Since $S^{\ast}$ is unknown, then the second term can be further upper bounded by 
\begin{align*}
	\|\bm{A}_{S \cap S^{\ast}, S^{\ast} \backslash S}\|_F \leq & ~ \sqrt{ \left\| \bm{A}_{ \{j_1\}, S^C}^{k - \tilde{k}} \right\|_2^2 + \cdots + \left\|\bm{A}_{\{ j_{\tilde{k}}\}, S^C}^{k - \tilde{k}} \right\|_2^2 } =: \text{ub}(S; \tilde{k}; S^C; k - \tilde{k}),
\end{align*}
where 
\begin{align*}
	\| \bm{A}_{\{j\}, S^C}^{l} \|_2^2 := \bm{A}_{j, i_1}^2 + \cdots + \bm{A}_{j, i_l}^2 \text{ with } |\bm{A}_{j, i_1}| \geq \cdots \geq |\bm{A}_{j, i_l}| \geq \ldots \text{ for all } i \in S^C,
\end{align*}
and $j_1, \ldots, j_{\tilde{k}}$ are indices satisfying: $
	\left\| \bm{A}_{j_1, S^C}^{k - \tilde{k}} \right\|_2^2 \geq \cdots \geq \left\|\bm{A}_{j_{\tilde{k}}, S^C}^{k - \tilde{k}} \right\|_2^2 \geq \cdots. $

Since $\tilde{k}$ is also not known, we arrive at our final upper bound $\text{ub}^{\text{sub-mat}}$ by considering all of its possibilities:
\begin{align*}
	\text{opt}^{\mathcal{F}} \leq &\max_{\tilde{k} = 0}^k \left\{ \text{ub}^{\text{CIP}}(\bm{A}_{S, S};\tilde{k}) + \sqrt{r}\cdot\text{ub}(S; \tilde{k}; S^C; k - \tilde{k}) + \text{Baseline-1}(\bm{A}_{S^C, S^C};k - \tilde{k}) \right\} \\
&=: \text{ub}^{\text{sub-mat}}.
\end{align*}

\subsubsection{Times for larger instances}

We set a more stringent time limit of 20 seconds for each \ref{eq:CIP-impl} used within the sub-matrix technique, since a number of these computations are required to compute $\text{ub}^{\text{sub-mat}}$. Again we did not set a time limit for the primal heuristic and just noted its running times as a function of the matrix size in Table~\ref{tab:lb-running-time}. 

\begin{table}[ht]
    \centering
    \begin{tabular}{ccccccccc} 
    \toprule
        size & \,$500 \times 500$\, & \,$1500 \times 1500$\, & \,$2000 \times 2000$\, \, \\ \midrule
		running time & $\leq 20$ min & $\leq 100$ min & $\leq 120$ min  \\
    \bottomrule
    \end{tabular} 
    \caption{Running time for primal heuristic}
    \label{tab:lb-running-time}
\end{table}


\subsubsection{Results on larger instances}

We compare the gap obtained by the upper bound $\text{ub}^{\text{sub-mat}}$ (\ref{eq:CIP-impl} plus sub-matrix technique) and compare it against that obtained by Baseline1 on the artificial and real instances with original sizes. These are reported on Tables~\ref{tab:artificial-instances-submatrix-tech} and~\ref{tab:real-instances-submatrix-tech-part1}.

	On the spiked covariance matrix artificial instances, we see that our dual bound $\text{ub}^{\text{sub-mat}}$ is typically orders of magnitude better than Baseline1 and is at most 0.35 for all instances. These results also illustrate that the sub-matrix ratio parameter can significantly impact the bound obtained by the sub-matrix technique.  

	On the real instances, we see from Table~\ref{tab:real-instances-submatrix-tech-part1} that on instances CovColon and Lymph our dual bound $\text{ub}^{\text{sub-mat}}$ performs slightly better than Baseline1 (except instance Lymph with parameters $(3, 10)$), and the gaps are overall less than 0.39. However, on instances Reddit1500 and Reddit2000 our dual bound $\text{ub}^{\text{sub-mat}}$ vastly outperforms Baseline1 on all settings of parameters. We remark that these are the largest instances in the experiments, which attest to the scalability of our proposed bound. 

\begin{table}[ht]
    \centering
  
    \begin{tabular}{|c|c||c|c|c||c|c|c|c|c|c|} 
    \nthickrule
       name & param $(r, k)$: &  $(2, 10)$ & $(2, 20)$ & $(2, 30)$ &  $(3, 10)$ & $(3, 20)$ & $(3, 30)$ \\ \nthickrule
       	
       	Artificial-10 & $m = 1.5$ & 0.527 & 0.151 & 0.25 &   0.366 & 0.1 & 0.169   \\
       	$500 \times 500$& $m = 2$ &  0.079 & 0.15 & 0.249 &   0.064 & 0.1 & 0.169   \\
       	& $m = 2.5$ &  0.079 & 0.15 & 0.248 &  0.064 & 0.099 & 0.168  \\ 
       	& $m = 5$ & 0.071 & 0.145 & 0.241 &  0.056 & 0.099 & 0.293  \\
       	& $m = 10$ & \bf 0.026 & \bf 0.002 & \bf 0.002 &   \bf 0.03 & \bf 0.003 & \bf 0.003 \\ 
       	& Baseline1 & 3.522 & 4.309 & 4.403 &  2.101 & 2.625 & 2.688  \\      	%
       	\nthickrule
       	
       	Artificial-20 & $m = 1.5$ & 2.397 & 0.566 & 0.268 &   1.629 & 0.384 & 0.186  \\
       	$500 \times 500$& $m = 2$ & 0.455 & 0.179 & 0.266 & 0.317 & 0.127 & 0.185  \\
       	& $m = 2.5$ & 0.606 & 0.178 & 0.265 &  0.463 & 0.126 & 0.184  \\ 
       	& $m = 5$ & 0.097 & 0.176 & 0.261 & \bf 0.078 & 0.124 & 0.346   \\
       	& $m = 10$ & \bf 0.073 & \bf 0.014 & \bf 0.009 &   0.139 & \bf 0.013 & \bf 0.008 \\
       	& Baseline1 & 3.58 & 7.838 & 8.251 &  2.097 & 4.942 & 5.216 \\
       	\nthickrule
       	
       	Artificial-30 & $m = 1.5$ & 3.515 & 0.595 & 0.65 &  2.071 & 0.406 & 0.425 \\
       	$500 \times 500$& $m = 2$ & 3.509 & 0.721 & 0.314 &   2.068 & 0.512 & 0.211  \\
       	& $m = 2.5$ &  2.304 & 0.709 & 0.312 &   1.586 & 0.511 & 0.209  \\
       	& $m = 5$ & 0.474 & 0.225 & 0.305 &  0.365 & 0.158 & 0.468 \\
       	& $m = 10$ & \bf 0.231 & \bf 0.026 & \bf 0.017 &   \bf 0.349 & \bf 0.154 & \bf 0.014  \\
       	& Baseline1 & 3.519 & 7.626 & 11.68 &   2.074 & 4.82 & 7.508 \\
       	\nthickrule
    \end{tabular}
    \caption{Gap values for artificial instances.}
    \label{tab:artificial-instances-submatrix-tech}
\end{table}

\begin{table}[ht]
    \centering
    \begin{tabular}{|c|c||c|c|c||c|c|c|c|c|c|} 
    \nthickrule
	name  & param $(r, k)$: &  $(2, 10)$ & $(2, 20)$ & $(2, 30)$ & $(3, 10)$ & $(3, 20)$ & $(3, 30)$  \\ \nthickrule
       	       	
       	CovColon & $m = 1.5$ & 0.054 & 0.112 & 0.128 &  0.05 & 0.08 & 0.092   \\
       	$500 \times 500$ & $m = 2$ & 0.051 & 0.107 & 0.126 &   0.062 & \bf 0.076 & 0.09 \\ 
       	& $m = 2.5$ & \bf 0.05 & \bf 0.104 & \bf 0.124 &   0.066 & 0.089 &\bf 0.088  \\ 
       	& $m = 5$ & 0.094 & 0.113 & 0.143 &  0.11 & 0.122 & 2.349 \\
       	& $m = 10$& 1.787 & 1.709 & 1.645 &   3.321 & 3.124 & 3.015  \\  
       	& Baseline1 & 0.063 & 0.118 & 0.133 &  \bf 0.049 & 0.086 & 0.097   \\
       	\nthickrule
       	
       	Lymph & $m = 1.5$ & 0.09 & 0.27 & 0.41 &    0.064 & 0.174 & 0.315   \\ 
       	$500 \times 500$ & $m = 2$ & \bf 0.078 & 0.267 & 0.406 &   0.103 & \bf 0.171 & 0.312  \\
       	& $m = 2.5$ & 0.104 & \bf 0.264 & 0.403 &  0.155 & 0.194 & \bf 0.309 \\ 
       	& $m = 5$ & 0.236 & 0.268 & \bf 0.388 &  0.2 & 0.296 & 2.698  \\ 
       	& $m = 10$ & 2.105 & 1.738 & 1.548 &  4.489 & 3.894 & 3.447  \\
       	& Baseline1 & 0.095 & 0.277 & 0.413 &  \bf 0.049 & 0.18 & 0.319  \\  
       	\nthickrule
       	
       	Reddit1500 & $m = 1.5$ & 0.687 & 0.95 & 0.8 &  0.39 & 0.625 & 0.677   \\ 
       	$1500 \times 1500$ & $m = 2$ & 0.683 & 0.94 & 0.749 &  0.387 & 0.617 & 0.632  \\ 
       	& $m = 2.5$ & 0.672 & 0.937 & \bf 0.727 &  0.377 & 0.614 & \bf 0.611  \\ 
       	& $m = 5$ & 0.426 & \bf 0.47 & 1.068 & 0.346 & \bf 0.393 & 1.307  \\ 
       	& $m = 10$ & \bf 0.384 & 0.927 & 1.075 & \bf 0.316 & 1.222 & 1.343 \\
       	& Baseline1 & 0.695 & 0.962 & 1.199 & 0.396 & 0.635 & 0.848  \\ 
       	\nthickrule
       	
       	Reddit2000 & $m = 1.5$ & 0.845 & 1.408 & 0.76 &  0.556 & 1.026 & 0.667  \\
       	$2000 \times 2000$ & $m = 2$ & 0.837 & 1.4 & 0.664 &  0.549 & 1.019 & 0.585  \\
       	& $m = 2.5$ & 0.827 & 1.396 & \bf 0.601 & 0.541 & 1.016 & \bf 0.538  \\
       	& $m = 5$ & 0.456 & \bf 0.436 & 1.52 &  0.395 & \bf 0.381 & 1.311  \\
       	& $m = 10$ & \bf 0.298 & 0.866 & 2.234 &  \bf 0.266 & 1.289 & 1.41   \\
       	& Baseline1 & 0.876 & 1.426 & 1.775 & 0.582 & 1.041 & 1.326  \\        	
  
    \nthickrule
    \end{tabular}
    \caption{Gap values for real instances.}
    \label{tab:real-instances-submatrix-tech-part1}
\end{table}



\section{Conclusion}
In this paper, we proposed a scheme for producing good primal feasible solutions and dual bounds for \ref{eq:Multi-SPCA} problem. The primal feasible solution is obtained from a monotonically improving heuristic for \ref{eq:Multi-SPCA} problem. We showed that the solutions produced by this algorithm are of very high quality by comparing the objective value of the solutions generated to upper bounds. These upper bounds are obtained using second-order cone IP relaxation designed in this paper. We also presented theoretical guarantees (affine guarantee) on the quality of the upper bounds produced by the second-order cone IP. The running times for both the primal algorithm and the dual bounding heuristic are very reasonable (less than $2$ hours for the $500\times 500$ instances and less than $3.5$ hours for the $2000\times 2000$ instance). These problems are quite challenging, and in some instances, we still need more techniques to close the gap. However, to the best of our knowledge, there are no comparable theoretical or computational results for solving model-free \ref{eq:Multi-SPCA}.

\section{Acknowledgements}
We would like to thank the anonymous reviewers for excellent comments that significantly improved the paper. In particular, the SDP presented in Section 2.2 and the heuristic method in Appendix B.3 have been suggested by the reviewers.

\medskip Marco Molinaro was supported in part by the Coordena\'c\~ao de Aperfei\'coamento de Pessoal de N\'ivel Superior (CAPES, Brasil) - Finance Code 001, by Bolsa de Produtividade em Pesquisa $\#3$12751/2021-4 from CNPq, FAPERJ grant ``Jovem Cientista do Nosso Estado'', and by the CAPES-PrInt program. Santanu S. Dey would like to gratefully acknowledge the support of the grant
N000141912323 from ONR.

\bibliographystyle{spmpsci}
\bibliography{refs}

\begin{thebibliography}{10}
\providecommand{\url}[1]{{#1}}
\providecommand{\urlprefix}{URL }
\expandafter\ifx\csname urlstyle\endcsname\relax
  \providecommand{\doi}[1]{DOI~\discretionary{}{}{}#1}\else
  \providecommand{\doi}{DOI~\discretionary{}{}{}\begingroup
  \urlstyle{rm}\Url}\fi

\bibitem{alizadeh2000distinct}
Alizadeh, A.A., Eisen, M.B., Davis, R.E., Ma, C., Lossos, I.S., Rosenwald, A.,
  Boldrick, J.C., Sabet, H., Tran, T., Yu, X., et~al.: Distinct types of
  diffuse large b-cell lymphoma identified by gene expression profiling.
\newblock Nature \textbf{403}(6769), 503 (2000)

\bibitem{alon1999broad}
Alon, U., Barkai, N., Notterman, D.A., Gish, K., Ybarra, S., Mack, D., Levine,
  A.J.: Broad patterns of gene expression revealed by clustering analysis of
  tumor and normal colon tissues probed by oligonucleotide arrays.
\newblock Proceedings of the National Academy of Sciences \textbf{96}(12),
  6745--6750 (1999)

\bibitem{asteris2015sparse}
Asteris, M., Papailiopoulos, D., Kyrillidis, A., Dimakis, A.G.: Sparse {PCA}
  via bipartite matchings.
\newblock In: Advances in Neural Information Processing Systems, pp. 766--774
  (2015)

\bibitem{asteris2011sparse}
Asteris, M., Papailiopoulos, D.S., Karystinos, G.N.: Sparse principal component
  of a rank-deficient matrix.
\newblock In: 2011 IEEE International Symposium on Information Theory
  Proceedings, pp. 673--677. IEEE (2011)

\bibitem{attouch2010proximal}
Attouch, H., Bolte, J., Redont, P., Soubeyran, A.: Proximal alternating
  minimization and projection methods for nonconvex problems: An approach based
  on the kurdyka-{\l}ojasiewicz inequality.
\newblock Mathematics of Operations Research \textbf{35}(2), 438--457 (2010)

\bibitem{berthet2013computational}
Berthet, Q., Rigollet, P.: Computational lower bounds for sparse pca.
\newblock arXiv preprint arXiv:1304.0828  (2013)

\bibitem{bolte2014proximal}
Bolte, J., Sabach, S., Teboulle, M.: Proximal alternating linearized
  minimization for nonconvex and nonsmooth problems.
\newblock Mathematical Programming \textbf{146}(1-2), 459--494 (2014)

\bibitem{boutsidis2011sparse}
Boutsidis, C., Drineas, P., Magdon-Ismail, M.: Sparse features for {PCA}-like
  linear regression.
\newblock In: Advances in Neural Information Processing Systems, pp. 2285--2293
  (2011)

\bibitem{burgel2010clinical}
Burgel, P.R., Paillasseur, J., Caillaud, D., Tillie-Leblond, I., Chanez, P.,
  Escamilla, R., Perez, T., Carr{\'e}, P., Roche, N., et~al.: Clinical {COPD}
  phenotypes: a novel approach using principal component and cluster analyses.
\newblock European Respiratory Journal \textbf{36}(3), 531--539 (2010)

\bibitem{cai2015optimal}
Cai, T., Ma, Z., Wu, Y.: Optimal estimation and rank detection for sparse
  spiked covariance matrices.
\newblock Probability theory and related fields \textbf{161}(3-4), 781--815
  (2015)

\bibitem{cai2013sparse}
Cai, T.T., Ma, Z., Wu, Y., et~al.: Sparse {PCA}: Optimal rates and adaptive
  estimation.
\newblock The Annals of Statistics \textbf{41}(6), 3074--3110 (2013)

\bibitem{chan2016approximability}
Chan, S.O., Papailliopoulos, D., Rubinstein, A.: On the approximability of
  sparse {PCA}.
\newblock In: Conference on Learning Theory, pp. 623--646 (2016)

\bibitem{chen2019alternating}
Chen, S., Ma, S., Xue, L., Zou, H.: An alternating manifold proximal gradient
  method for sparse {PCA} and sparse {CCA}.
\newblock arXiv preprint arXiv:1903.11576  (2019)

\bibitem{d2014approximation}
d'Aspremont, A., Bach, F., El~Ghaoui, L.: Approximation bounds for sparse
  principal component analysis.
\newblock Mathematical Programming \textbf{148}(1-2), 89--110 (2014)

\bibitem{d2008optimal}
d'Aspremont, A., Bach, F., Ghaoui, L.E.: Optimal solutions for sparse principal
  component analysis.
\newblock Journal of Machine Learning Research \textbf{9}(Jul), 1269--1294
  (2008)

\bibitem{d2005direct}
d'Aspremont, A., Ghaoui, L.E., Jordan, M.I., Lanckriet, G.R.: A direct
  formulation for sparse {PCA} using semidefinite programming.
\newblock In: Advances in neural information processing systems, pp. 41--48
  (2005)

\bibitem{alberto2019sparse}
Del~Pia, A.: Sparse {PCA} on fixed-rank matrices.
\newblock http://www.optimization-online.org/DB\_HTML/2019/07/7307.html  (2019)

\bibitem{deshpande2016sparse}
Deshpande, Y., Montanari, A.: Sparse {PCA} via covariance thresholding.
\newblock The Journal of Machine Learning Research \textbf{17}(1), 4913--4953
  (2016)

\bibitem{dey2018convex}
Dey, S.S., Mazumder, R., Wang, G.: A convex integer programming approach for
  optimal sparse pca.
\newblock arXiv preprint arXiv:1810.09062  (2018)

\bibitem{erichson2018sparse}
Erichson, N.B., Zheng, P., Manohar, K., Brunton, S.L., Kutz, J.N., Aravkin,
  A.Y.: Sparse principal component analysis via variable projection.
\newblock arXiv preprint arXiv:1804.00341  (2018)

\bibitem{gallivan2010note}
Gallivan, K.A., Absil, P.: Note on the convex hull of the stiefel manifold.
\newblock Technical note  (2010)

\bibitem{gu2014sparse}
Gu, Q., Wang, Z., Liu, H.: Sparse {PCA} with oracle property.
\newblock In: Advances in neural information processing systems, pp. 1529--1537
  (2014)

\bibitem{hiriart2012fundamentals}
Hiriart-Urruty, J.B., Lemar{\'e}chal, C.: Fundamentals of convex analysis.
\newblock Springer Science \& Business Media (2012)

\bibitem{johnstone2009sparse}
Johnstone, I.M., Lu, A.Y.: Sparse principal components analysis.
\newblock arXiv preprint arXiv:0901.4392  (2009)

\bibitem{jolliffe2016principal}
Jolliffe, I.T., Cadima, J.: Principal component analysis: a review and recent
  developments.
\newblock Philosophical Transactions of the Royal Society A: Mathematical,
  Physical and Engineering Sciences \textbf{374}(2065), 20150202 (2016)

\bibitem{jolliffe2003modified}
Jolliffe, I.T., Trendafilov, N.T., Uddin, M.: A modified principal component
  technique based on the {LASSO}.
\newblock Journal of computational and Graphical Statistics \textbf{12}(3),
  531--547 (2003)

\bibitem{journee2010generalized}
Journ{\'e}e, M., Nesterov, Y., Richt{\'a}rik, P., Sepulchre, R.: Generalized
  power method for sparse principal component analysis.
\newblock Journal of Machine Learning Research \textbf{11}(Feb), 517--553
  (2010)

\bibitem{kannan2017randomized}
Kannan, R., Vempala, S.: Randomized algorithms in numerical linear algebra.
\newblock Acta Numerica \textbf{26}, 95 (2017)

\bibitem{kim2019convexification}
Kim, J., Tawarmalani, M., Richard, J.P.P.: Convexification of
  permutation-invariant sets and applications.
\newblock arXiv preprint arXiv:1910.02573  (2019)

\bibitem{krauthgamer2015semidefinite}
Krauthgamer, R., Nadler, B., Vilenchik, D., et~al.: Do semidefinite relaxations
  solve sparse {PCA} up to the information limit?
\newblock The Annals of Statistics \textbf{43}(3), 1300--1322 (2015)

\bibitem{lei2015sparsistency}
Lei, J., Vu, V.Q., et~al.: Sparsistency and agnostic inference in sparse {PCA}.
\newblock The Annals of Statistics \textbf{43}(1), 299--322 (2015)

\bibitem{ma2013alternating}
Ma, S.: Alternating direction method of multipliers for sparse principal
  component analysis.
\newblock Journal of the Operations Research Society of China \textbf{1}(2),
  253--274 (2013)

\bibitem{ma2015sum}
Ma, T., Wigderson, A.: Sum-of-squares lower bounds for sparse pca.
\newblock In: Advances in Neural Information Processing Systems, pp. 1612--1620
  (2015)

\bibitem{mackey2009deflation}
Mackey, L.W.: Deflation methods for sparse {PCA}.
\newblock In: Advances in neural information processing systems, pp. 1017--1024
  (2009)

\bibitem{magdon2017np}
Magdon-Ismail, M.: {NP}-hardness and inapproximability of sparse {PCA}.
\newblock Information Processing Letters \textbf{126}, 35--38 (2017)

\bibitem{mitzenmacher2017probability}
Mitzenmacher, M., Upfal, E.: Probability and computing: Randomization and
  probabilistic techniques in algorithms and data analysis.
\newblock Cambridge university press (2017)

\bibitem{papailiopoulos2013sparse}
Papailiopoulos, D., Dimakis, A., Korokythakis, S.: Sparse {PCA} through
  low-rank approximations.
\newblock In: International Conference on Machine Learning, pp. 747--755 (2013)

\bibitem{pietsch1978operator}
Pietsch, A.: Operator ideals, vol.~16.
\newblock Deutscher Verlag der Wissenschaften (1978)

\bibitem{probel2011technical}
PROBEL, C.J., TROPP, J.A.: Technical report no. 2011-02 august 2011  (2011)

\bibitem{sigg2008expectation}
Sigg, C.D., Buhmann, J.M.: Expectation-maximization for sparse and non-negative
  {PCA}.
\newblock In: Proceedings of the 25th international conference on Machine
  learning, pp. 960--967. ACM (2008)

\bibitem{steinberg2005computation}
Steinberg, D.: Computation of matrix norms with applications to robust
  optimization.
\newblock Research thesis, Technion-Israel University of Technology \textbf{2}
  (2005)

\bibitem{tropp2009column}
Tropp, J.A.: Column subset selection, matrix factorization, and eigenvalue
  optimization.
\newblock In: Proceedings of the twentieth annual ACM-SIAM symposium on
  Discrete algorithms, pp. 978--986. SIAM (2009)

\bibitem{tropp2012user}
Tropp, J.A.: User-friendly tail bounds for sums of random matrices.
\newblock Foundations of computational mathematics \textbf{12}(4), 389--434
  (2012)

\bibitem{vu2012minimax}
Vu, V., Lei, J.: Minimax rates of estimation for sparse {PCA} in high
  dimensions.
\newblock In: Artificial intelligence and statistics, pp. 1278--1286 (2012)

\bibitem{vu2013fantope}
Vu, V.Q., Cho, J., Lei, J., Rohe, K.: Fantope projection and selection: A
  near-optimal convex relaxation of sparse {PCA}.
\newblock In: Advances in neural information processing systems, pp. 2670--2678
  (2013)

\bibitem{wangupper}
Wang, G., Dey, S.: Upper bounds for model-free row-sparse principal component
  analysis.
\newblock In: Proceedings of the International Conference on Machine Learning
  (2020)

\bibitem{wang2014tighten}
Wang, Z., Lu, H., Liu, H.: Tighten after relax: Minimax-optimal sparse {PCA} in
  polynomial time.
\newblock In: Advances in neural information processing systems, pp. 3383--3391
  (2014)

\bibitem{wolsey1999integer}
Wolsey, L.A., Nemhauser, G.L.: Integer and combinatorial optimization, vol.~55.
\newblock John Wiley \& Sons (1999)

\bibitem{yeung2001principal}
Yeung, K.Y., Ruzzo, W.L.: Principal component analysis for clustering gene
  expression data.
\newblock Bioinformatics \textbf{17}(9), 763--774 (2001)

\bibitem{yongchun2020exact}
Yongchun~Li, W.X.: Exact and approximation algorithms for sparse {PCA}.
\newblock http://www.optimization-online.org/DB\_HTML/2020/05/7802.html  (2020)

\bibitem{yuan2013truncated}
Yuan, X.T., Zhang, T.: Truncated power method for sparse eigenvalue problems.
\newblock Journal of Machine Learning Research \textbf{14}(Apr), 899--925
  (2013)

\bibitem{zhang2012sparse}
Zhang, Y., d'Aspremont, A., El~Ghaoui, L.: Sparse {PCA}: Convex relaxations,
  algorithms and applications.
\newblock In: Handbook on Semidefinite, Conic and Polynomial Optimization, pp.
  915--940. Springer (2012)

\bibitem{zou2006sparse}
Zou, H., Hastie, T., Tibshirani, R.: Sparse principal component analysis.
\newblock Journal of computational and graphical statistics \textbf{15}(2),
  265--286 (2006)

\end{thebibliography}

\vspace{1cm}
\appendix
\noindent {\LARGE \bf Appendix}

\section{Additional concentration inequalities} \label{app:conc}

	We need the standard multiplicative Chernoff bound (see Theorem 4.4~\cite{mitzenmacher2017probability}).
	\begin{lemma}[Chernoff Bound] \label{lemma:chernoff}
		Let $X_1,\ldots,X_n$ be independent random variables taking values in $[0,1]$. Then for any $\delta > 0$ we have 
		\begin{align*}
			\Pr\bigg(\sum_i X_i > (1+\delta) \mu \bigg) \,<\, \bigg(\frac{e}{1+\delta}\bigg)^{(1+\delta) \mu},
		\end{align*}
		where $\mu = \E \sum_i X_i$. 
	\end{lemma}

	We also need the one-sided Chebychev inequality, see for example Exercise 3.18 of~\cite{mitzenmacher2017probability}.
	
	\begin{lemma}[One-sided Chebychev] \label{lemma:cheby}
		For any random variable $X$ with finite first and second moments
		\begin{align*}
			\Pr\bigg( X \le \E X - t \bigg) \,\le\, \frac{\Var(X)}{\Var(X) + t^2}.
		\end{align*}
	\end{lemma}

{\color{black}
\section{Scaling invariance of Theorem \ref{thm:CIP}} \label{app:invariance}

Given any data matrix $\bm{X} \in \mathbb{R}^{d \times M}$ with $M$ samples, the sample covariance matrix is $\bm{A} = \frac{1}{M} \bm{X}\bm{X}^{\top}$. Theorem~\ref{thm:CIP} shows that
\begin{align*}
	\text{opt}^{\mathcal{F}}(\bm{A}) \leq \text{ub}^{\mathcal{CR}i}(\bm{A}) \leq \rho_{\mathcal{CR}i}^2 \cdot \text{opt}^{\mathcal{F}}(\bm{A}) + \underbrace{\sum_{j = 1}^d \frac{r \theta_j^2 \lambda_j(\bm{A})}{4N^2}}_{\text{additive term}, ~ =: \text{add}(\bm{A})} =: \text{aff}(\bm{A}).  
\end{align*}
Note that rescaling the data matrix $\bm{X}$ to $\tilde{\bm{X}} = c \cdot \bm{X}$ for any constant $c > 0$ does not change the approximation ratios $\rho_{\mathcal{CR}i}^2$ for $i \in \{1, 1', 2\}$ and changes the terms $\textup{opt}^{\mathcal{F}}(\bm{A})$ and $\textup{add}(\bm{A})$ quadratically: letting $\tilde{\bm{A}} := \frac{1}{M} \tilde{\bm{X}}\tilde{\bm{X}}^{\top} = c^2 \bm{A}$,
\begin{align*}
	& \text{opt}^{\mathcal{F}}(\tilde{\bm{A}}) = \max_{\bm{V}^{\top} \bm{V} = \bm{I}^r, \|\bm{V}\|_0 \leq k} \text{Tr}(\bm{V}^{\top}\tilde{\bm{A}} \bm{V}) = \max_{\bm{V}^{\top} \bm{V} = \bm{I}^r, \|\bm{V}\|_0 \leq k} c^2 \cdot \text{Tr}(\bm{V}^{\top} \bm{A} \bm{V}) = c^2 \cdot \text{opt}^{\mathcal{F}}(\bm{A}), \\
	& \text{add}(\tilde{\bm{A}}) = \sum_{j = 1}^d \frac{r \theta_j^2 \lambda_j(\tilde{\bm{A}})}{4N^2} = \sum_{j = 1}^d \frac{r \theta_j^2 \lambda_j(c^2 \bm{A})}{4N^2} = c^2 \cdot \sum_{j = 1}^d \frac{r \theta_j^2 \lambda_j(\bm{A})}{4N^2} = c^2 \cdot \text{add}(\bm{A}).
\end{align*}
In particular, this implies that the \emph{effective multiplicative ratio (emr)} between $\text{opt}^{\mathcal{F}}(\bm{A})$  and the affine upper bound $\textup{aff}(\bm{A})$ is invariant under rescaling:
\begin{align*}
	\text{emr}(\bm{A}) := \frac{\text{aff}(\bm{A})}{\text{opt}^{\mathcal{F}}(\bm{A})} = & ~ \frac{\rho_{\mathcal{CR}i}^2 \cdot \text{opt}^{\mathcal{F}}(\bm{A}) + \text{add}(\bm{A})}{\text{opt}^{\mathcal{F}}(\bm{A})} \\
	= & ~ \frac{\rho_{\mathcal{CR}i}^2 \cdot c^2 \cdot \text{opt}^{\mathcal{F}}(\bm{A}) + c^2 \cdot \text{add}(\bm{A})}{ c^2 \cdot \text{opt}^{\mathcal{F}}(\bm{A})} \\
	= & ~ \frac{\rho_{\mathcal{CR}i}^2 \cdot \text{opt}^{\mathcal{F}}(\tilde{\bm{A}}) + \text{add}(\tilde{\bm{A}})}{ \text{opt}^{\mathcal{F}}(\tilde{\bm{A}})} \\
	= & ~ \frac{\text{aff}(\tilde{\bm{A}})}{\text{opt}^{\mathcal{F}}(\tilde{\bm{A}})} = \text{emr}(\tilde{\bm{A}}).
\end{align*}
}


\section{Greedy heuristic for \ref{eq:Multi-SPCA}}
	
	\subsection{Complete pseudocode} \label{app:pseudo}
	
\begin{algorithm}[ht] 
  \textbf{Input:} Covariance matrix $\bm{A}$, sparsity parameter $k$, number of maximum iterations $T$\\
  \textbf{Output:} A feasible solution $\bm{V}$ for \ref{eq:Multi-SPCA}.  
  \begin{algorithmic}
  \STATE \textbf{Initialize} with $S_0 \subseteq [d]$
\STATE Compute eigendecomposition of $A_{S_0}$: $\bm{A}_{S_0,S_0} = \bm{U}_{S_0} \bm{\Lambda}_{S_0} \bm{U}_{S_t}^{\top}$, $\bm{V}_{S_0} = (\bm{U}_{S_0})_{\star, [r]}$
  \FOR{$t = 1, \ldots, T$}   
    	 \STATE Compute the leaving candidate $j^{\text{out}} := \argmin_{j \in S_{t - 1}} \Delta^{\oout}_j$  	 
		 \STATE Compute the entering candidate $j^{\text{in}} := \argmax_{j \in S^C_{t - 1}} \Delta^{\iin}_{j}$    	 
    	 \IF{$\Delta^{\iin}_{j^{\text{in}}} > \Delta^{\oout}_{j^{\text{out}}}$}
    	 	\STATE Set $S_t := S_{t - 1} - \{j^{\text{out}}\} + \{j^{\text{in}}\}$    	 	
    	 	\STATE Compute the eigenvalue decomposition $(\bm{A}^{1/2})_{S_{t}} = \bm{U}_{S_{t}} \bm{\Lambda}_{S_t} \bm{U}_{S_t}^{\top}$
    	 	\STATE Set $\bm{V}^t_{S_t} = (\bm{U}_{S_t})_{\star, [r]}$
    	 \ELSE    	 
    	 	\STATE \textbf{Return} the matrix $\bm{V}$ where in rows $S_{t-1}$ equals $\bm{V}^{t-1}_{S_{t-1}}$ (i.e., $\bm{V}_{S_{t-1}} = \bm{V}^{t-1}_{S_{t-1}}$) and in rows $S^C_{t-1}$ equals zero
    	 \ENDIF    	 
   \ENDFOR
   \end{algorithmic}
\caption{Modified greedy neighborhood search}
\label{algo:LSM}
\end{algorithm}


	\subsection{Proof of Lemma \ref{lemma:primal}} \label{app:monoProof}
	
	By optimality of $\bm{V}^t_{S_t}$ we can see that $f(S_t) = f(S_t, \bm{V}^t_{S_t})$ for all $t$. Thus, letting $\bm{G}_t := \bm{I}^k - \bm{V}^t_{S_{t}} (\bm{V}_{S_{t}}^t)^{\top}$ to simplify the notation, we have
\begin{align*} 
f(S_{t-1}) = f(S_{t-1}, \bm{V}^{t-1}_{S_{t-1}})	&= \left\|\bm{G}_t\, (\bm{A}^{1/2})_{S_{t - 1}} \right\|_F^2 + \sum_{j \in S_{t - 1}^C} \left\|(\bm{A}^{1/2})_{j} \right\|_2^2 \\
	&= ~ \left\|\bm{G}_t\, \bm{A}^{1/2}_{S_{t}} \right\|_F^2 + \sum_{j \in S_t^C} \left\|\bm{A}^{1/2}_{j} \right\| _2^2 + \underbrace{ \Delta^{\iin}_{j^{\text{in}}} - \Delta^{\oout}_{j^{\text{out}}} }_{> 0} \\
	&> ~ \left\|\bm{G}_t \bm{A}^{1/2}_{S_{t}} \right\|_F^2 + \sum_{j \in S_t^C} \left\|\bm{A}^{1/2}_{j} \right\| _2^2\\
	&= f(S_t, \bm{V}^t_{S_t}) = f(S_t).
\end{align*}

{\color{black}
\subsection{Primal Heuristic Algorithm For Near-Maximal Determinant} \label{app:PH-det}
Here we present another primal heuristic algorithm that finds a principal submatrix whose determinant is near-maximal. 
\begin{algorithm}[ht] 
  \textbf{Input:} Covariance $\bm{A}$, sparsity $k$, parameter $r$.
  \begin{algorithmic}
  \STATE \textbf{Initialize} with the support set $S = \{\}$. 
  \FOR{$t = 1, \ldots, k$}
  \STATE Compute $i_t := \argmax_{i \in [d] \backslash S} \text{det}\left( \bm{A}_{S \cup \{i\}, S \cup \{i\}}\right)$. 
  \STATE Set $S := S \cup \{i_t\}$.
  \ENDFOR
  \end{algorithmic}
  \textbf{Output:} Support set $S$ and lower bound $\text{lb}_{\text{GNS}} := \argmax_{\bm{V}^{\top} \bm{V} = \bm{I}^r} \text{Tr}(\bm{V}^{\top} \bm{A}_{S, S} \bm{V})$. 
\caption{Greedy heuristic of finding a submatrix with near-maximal determinant}
\label{algo:GH-det}
\end{algorithm}

We compare the performance of the greedy neighborhood search Algorithm~\ref{algo:LSM}, with the greedy heuristic (\text{GH}) \ref{algo:GH-det} in Table~\ref{tab:artificial-instances-primal} and Table~\ref{tab:real-instances-primal} where we report the relative gap defined as $\text{Gap} := \frac{\text{lb}_{\text{GH}}}{\text{lb}_{\text{GNS}}}$, where $\text{lb}_{\text{GH}}, \text{lb}_{\text{GNS}}$ denote the primal lower bounds of sparse PCA obtained from the greedy heuristic and the greedy neighborhood search respectively. 

\begin{table}[ht]
    \centering
  
    \begin{tabular}{|c|c||c|c|c||c|c|c|c|c|c|} 
    \nthickrule
       name (size)& param $(r, k)$: &  $(2, 10)$ & $(2, 20)$ & $(2, 30)$ &  $(3, 10)$ & $(3, 20)$ & $(3, 30)$ \\ \nthickrule
       	
       	Artificial-10 (500)& Gap:  & 0.972 & 0.939 & 0.939 & 0.989 & 0.951 & 0.951 \\
         \specialrule{.1em}{.05em}{.05em} 
        
        Artificial-20 (500)& Gap: & 0.985 & 1.0 & 0.994 & 0.992 & 1.0 & 0.993  \\
         \specialrule{.1em}{.05em}{.05em}

        Artificial-30 (500)& Gap:& 0.97 & 0.979 & 1.0 & 0.984 & 0.983 & 1.0   \\
         \specialrule{.1em}{.05em}{.05em} 
    \end{tabular}
    \caption{Comparison between \text{GH} Algo~\ref{algo:GH-det} and \text{GNS} Algo~\ref{algo:LSM} on artificial instances where $\text{Gap} := \frac{\text{lb}_{\text{GH}}}{\text{lb}_{\text{GNS}}}$. }
    \label{tab:artificial-instances-primal}
\end{table}

\begin{table}[ht]
    \centering
  
    \begin{tabular}{|c|c||c|c|c||c|c|c|c|c|c|} 
    \nthickrule
       name (size)& param $(r, k)$: &  $(2, 10)$ & $(2, 20)$ & $(2, 30)$ &  $(3, 10)$ & $(3, 20)$ & $(3, 30)$ \\ \nthickrule
       	
       	CovColon (500)& Gap: & 0.48 & 0.383 & 0.362 & 0.522 & 0.416 & 0.387 \\
         \specialrule{.1em}{.05em}{.05em} 
        
        Lymp (500)& Gap:  & 0.426 & 0.384 & 0.366 & 0.496 & 0.442 & 0.424  \\
         \specialrule{.1em}{.05em}{.05em} 
         
        Reddit1500 (1500)& Gap: & 0.884 & 0.832 & 0.807 & 0.899 & 0.848 & 0.834  \\
         \specialrule{.1em}{.05em}{.05em} 
         
        Reddit2000 (2000)& Gap: & 0.931 & 0.912 & 0.908 & 0.936 & 0.91 & 0.9 \\
         \specialrule{.1em}{.05em}{.05em} 
    \end{tabular}
    \caption{Comparison between \text{GH} Algo~\ref{algo:GH-det} and \text{GNS} Algo~\ref{algo:LSM} on real instances where $\text{Gap} := \frac{\text{lb}_{\text{GH}}}{\text{lb}_{\text{GNS}}}$.}
    \label{tab:real-instances-primal}
\end{table}

Based on the numerical results in Table~\ref{tab:artificial-instances-primal} and Table~\ref{tab:real-instances-primal}, the greedy neighborhood search (GNS) algorithm outperforms the greedy heuristic (GH) in every instance. 

} 


\section{Techniques for reducing the running time of \ref{eq:CIP}} \label{app:reduce-running-time}
In practice, we want to reduce the running time of \ref{eq:CIP}. Here are the techniques that we used to enhance the efficiency in practice.

\subsection{Threshold}\label{app:threshold} 
The first technique is to reduce the number of SOS-II constraints in the set PLA. Let $\lambda_{\mathrm{TH}}$ be a threshold parameter that splits the eigenvalues $\{\lambda_j\}_{j = 1}^d$ of sample covariance matrix $\bm{A}$ into two parts $J^+ = \{j: \lambda_j > \lambda_{\mathrm{TH}} \}$ and $J^- = \{j: \lambda_j \leq \lambda_{\mathrm{TH}} \}$. The objective function $\text{Tr} \left( \bm{V}^{\top} \bm{A} \bm{V} \right)$ satisfies 
\begin{align*}
	\text{Tr} \left( \bm{V}^{\top} \bm{A} \bm{V} \right) = \sum_{j \in J^+} (\lambda_j - \lambda_{\mathrm{TH}}) \sum_{i = 1}^r g_{ji}^2 + \sum_{j \in J^-} (\lambda_j - \lambda_{\mathrm{TH}}) \sum_{i = 1}^r g_{ji}^2 + \lambda_{\mathrm{TH}} \sum_{j = 1}^d \sum_{i = 1}^r g_{ji}^2, 
\end{align*}
in which the first term is convex, the second term is concave, and the third term satisfies
\begin{align*}
	\lambda_{\mathrm{TH}} \sum_{j = 1}^d \sum_{i = 1}^r g_{ji}^2 \leq r \lambda_{\mathrm{TH}} \tag{threshold-term} \label{eq:threshold-term}
\end{align*}
due to $\sum_{j = 1}^d \sum_{i = 1}^r g_{ji}^2 \leq r$. Since maximizing a concave function is equivalent to convex optimization, we replace the second term by a new auxiliary variable $s$ and the third term by its upper bound $r \lambda_{\mathrm{TH}}$ such that 
\begin{align}
	\text{Tr} \left( \bm{V}^{\top} \bm{A} \bm{V} \right) \leq & ~ \sum_{j \in J^+} (\lambda_j - \lambda_{\mathrm{TH}}) \sum_{i = 1}^r g_{ji}^2 - s + r \lambda_{\mathrm{TH}} \tag{threshold-tech} \label{eq:threshold-tech} 
\end{align}
where
\begin{align*}
	s \geq \sum_{j \in J^-} \underbrace{(\lambda_{\mathrm{TH}} - \lambda_j)}_{\geq 0} \sum_{i = 1}^r g_{ji}^2 \tag{s-var} \label{eq:s-var}
\end{align*}
is a convex constraint. 
We select a value of $\lambda_{\mathrm{TH}}$ so that $|J^+| = 3$. Therefore, it is sufficient to construct a piecewise-linear upper approximation for the quadratic terms $g_{ji}^2$ in the first term with $j \in J^+$, i.e., constraint set $\text{PLA}([J^+] \times [r])$. 
We thus, greatly reduce the number of SOS-II constraints from $\mathcal{O}( d \times r )$ to $\mathcal{O}( |J^+| \times r )$, i.e. in our experiments to $3r$ SOS-II constraints.

\subsection{Cutting planes}
Similar to classical integer programming, we can incorporate additional cutting planes to improve the efficiency. 
 
\textbf{Cutting plane for sparsity:} 
The first family of cutting-planes is obtained as follows: Since $\|\bm{V}\|_0 \leq k$ and $\bm{v}_1, \ldots, \bm{v}_r$ are orthogonal, by Bessel inequality, we have
\begin{align*}
	& \sum_{i = 1}^r g_{ji}^2 = \sum_{i = 1}^r (\bm{a}_j^{\top} \bm{v}_i)^2 = \bm{a}_j^{\top} \bm{V} \bm{V}^{\top} \bm{a}_j \leq \theta_j^2, \tag{sparse-g} \label{eq:sparse-g} \\
	& \sum_{i = 1}^r \xi_{ji} \leq \theta_j^2 \left(1 + \frac{r}{4N^2} \right). \tag{sparse-xi} \label{eq:sparse-xi} 
\end{align*}
We call these above cuts--sparse cut since $\theta_j$ is obtained from the row sparsity parameter $k$. 

\textbf{Cutting plane from objective value:} 
The second type of cutting plane is based on the property: for any symmetric matrix, the sum of its diagonal entries are equal to the sum of its eigenvalues. Let $\bm{A}_{j_1, j_1}, \ldots, \bm{A}_{j_k, j_k}$ be the largest $k$ diagonal entries of the sample covariance matrix $\bm{A}$, we have
\begin{proposition}
	The following are valid cuts for \ref{eq:Multi-SPCA}: 
	\begin{align*}
		\sum_{j = 1}^d \lambda_j \sum_{i = 1}^r g_{ji}^2 \leq \bm{A}_{j_1, j_1} + \cdots + \bm{A}_{j_k, j_k}. \tag{cut-g} \label{eq:obj-cut-g}
	\end{align*}
	When the splitting points $\{\gamma_{ji}^{\ell}\}_{\ell = - N}^N$ in SOS-II are set to be $\gamma_{ji}^{\ell} = \frac{\ell}{N} \cdot \theta_j$, we have:
	\begin{align*}\tag{cut-xi} \label{eq:obj-cut-xi}
\begin{array}{rcl}
		\sum_{j \in J^+} (\lambda_j - \lambda_{\mathrm{TH}})\sum_{i = 1}^r \xi_{ji} - s + g \lambda_{\mathrm{TH}} &\leq& \bm{A}_{j_1, j_1} + \cdots + \bm{A}_{j_k, j_k} + \sum_{j \in J^+} \frac{r (\lambda_j - \phi) \theta_j^2}{4 N^2} \\
g &\geq &\sum_{j = 1}^d \sum_{i = 1}^r g_{ji}^2.
\end{array}
	\end{align*}
\end{proposition}

\subsection{Implemented version of \ref{eq:CIP}}
Thus the implemented version of \ref{eq:CIP} is
\begin{align*}
	\begin{array}{llll}
		\max & \sum_{j \in J^+} (\lambda_j - \lambda_{\mathrm{LB}}) \sum_{i = 1}^r \xi_{ji} - s + r \lambda_{\mathrm{LB}} \\
		\text{s.t} & \bm{V} \in \mathcal{CR}2 \\
		& (g, \xi, \eta) \in \text{PLA}([J^+] \times [r]) \\
		& \text{(\ref{eq:s-var}), (\ref{eq:sparse-g}), (\ref{eq:sparse-xi}), (\ref{eq:obj-cut-g}), (\ref{eq:obj-cut-xi})}
	\end{array} \tag{CIP-impl} \label{eq:CIP-impl}
\end{align*}

\subsection{Submatrix technique}
\begin{proposition}\label{prop:submat}
Let $X \in \mathbb{R}^{m \times n}$ and let $\theta$ be defined as
\begin{eqnarray*}
\theta:= 2\textup{max}_{\bm{V}^1 \in \mathbb{R}^{m \times r}, \bm{V}^2 \in \mathbb{R}^{n \times r}} ~ 2 \textup{Tr}\left( (\bm{V}^1)^{\top} \bm{X} \bm{V}^2 \right) \text{ s.t. } (\bm{V}^1)^{\top} \bm{V}^1 + (\bm{V}^2)^{\top} \bm{V}^2 = \bm{I}^r, 
\end{eqnarray*}
then $\theta \leq \sqrt{r} \|X\|_F$
\end{proposition}
\begin{proof}
\begin{align*}
	& ~  \max_{\bm{V}^1, \bm{V}^2} ~ 2 \textup{Tr}\left( (\bm{V}^1)^{\top} \bm{X} \bm{V}^2 \right) \text{ s.t. } (\bm{V}^1)^{\top} \bm{V}^1 + (\bm{V}^2)^{\top} \bm{V}^2 = \bm{I}^r, \\
	\Leftrightarrow ~ & ~ \max_{\bm{V}^1, \bm{V}^2} ~  \textup{Tr}\left( 
	\begin{pmatrix}
		\bm{V}^1)^{\top} & (\bm{V}^2)^{\top}  
	\end{pmatrix}
	\begin{pmatrix}
		0 & \bm{X} \\
		\bm{X}^{\top} & 0 \\
	\end{pmatrix}
	\begin{pmatrix}
		\bm{V}^1 \\ 
		\bm{V}^2 \\
	\end{pmatrix}
	\right) \text{ s.t. } (\bm{V}^1)^{\top} \bm{V}^1 + (\bm{V}^2)^{\top} \bm{V}^2 = \bm{I}^r, \\
	\Leftrightarrow ~ & ~ \max_{\bm{V}} ~  \text{Tr}\left( 
	\bm{V}^{\top}
	\begin{pmatrix}
		0 & \bm{X} \\
		\bm{X}^{\top} & 0 \\
	\end{pmatrix}
	\bm{V}
	\right) \text{ s.t. } \bm{V}^{\top} \bm{V} = \bm{I}^r.
\end{align*}
Note that the final maximization problem is equal to 
\begin{align*}
	& ~ \max_{\bm{V}} ~  \text{Tr}\left( 
	\bm{V}^{\top}
	\begin{pmatrix}
		0 & \bm{X} \\
		\bm{X}^{\top} & 0 \\
	\end{pmatrix}
	\bm{V}
	\right) \text{ s.t. } \bm{V}^{\top} \bm{V} = \bm{I}^r \\
	\leq & ~ \sum_{i = 1}^r \lambda_i \left(
	\begin{pmatrix}
		0 & \bm{X} \\
		\bm{X}^{\top} & 0 \\
	\end{pmatrix}
	\right), 
\end{align*}
Next we verify that the eigenvalues of 
$$ \left(\begin{array}{cc} 0 & X \\ X^{\top}& 0 \end{array}\right)$$ are $\pm$ singular values of $X$: Let $X = U\Sigma W^{\top}$. In particular, note that:
\begin{eqnarray*}
\begin{array}{rcccl}
\left(\begin{array}{cc} 0 & U\Sigma W^{\top} \\ W\Sigma U^{\top}& 0 \end{array}\right)\left[\begin{array}{c} u_i \\ w_i \end{array}\right] &=& \left[\begin{array}{c} U\Sigma e_i \\W\Sigma e_i\end{array}\right] &=& \sigma_i(X)\left[\begin{array}{c} u_i \\ w_i \end{array}\right]\\
\left(\begin{array}{cc} 0 & U\Sigma W^{\top} \\ W\Sigma U^{\top}& 0 \end{array}\right)\left[\begin{array}{c} u_i \\ -w_i \end{array}\right] &=& \left[\begin{array}{c} -U\Sigma e_i \\W\Sigma e_i\end{array}\right] &=& -\sigma_i(X)\left[\begin{array}{c} u_i \\ -w_i \end{array}\right].
\end{array}
\end{eqnarray*}
Therefore, we have 
\begin{align*}
	\sum_{i = 1}^r \lambda_i \left(
	\begin{pmatrix}
		0 & \bm{X} \\
		\bm{X}^{\top} & 0 \\
	\end{pmatrix}
	\right) = \sum_{i = 1}^r \sigma_i(\bm{X}) \leq \sqrt{r} \|\bm{X}\|_F. 
\end{align*}

\end{proof}



\end{document}